\documentclass[a4paper,reqno,11pt,oneside]{amsart}
\usepackage{amssymb, amstext, amsmath, amscd, amsthm, amsfonts, latexsym, mathrsfs, graphicx, pstricks, lscape, comment}
\usepackage[colorlinks=true,citecolor=blue,linkcolor=blue]{hyperref}
\usepackage[top=25truemm,bottom=25truemm, left=25truemm, right=25truemm]{geometry}
\usepackage[all]{xy}

\makeatletter
\@namedef{subjclassname@2020}{%
  \textup{2020} Mathematics Subject Classification}
\makeatother

\theoremstyle{plain}
\newtheorem{thm}{Theorem}[section]
\newtheorem{cor}[thm]{Corollary}
\newtheorem*{thm*}{Theorem}
\newtheorem*{cor*}{Corollary}
\newtheorem{lem}[thm]{Lemma}
\newtheorem{prop}[thm]{Proposition}

\theoremstyle{definition}
\newtheorem{dfn}[thm]{Definition}
\newtheorem{ex}[thm]{Example}
\newtheorem{ques}[thm]{Question}

\newtheorem{cvn}[thm]{Convention}
\newtheorem{rem}[thm]{Remark}

\newcommand{\NN}{\mathbb{N}}
\newcommand{\ZZ}{\mathbb{Z}}

\newcommand{\PP}{\mathbb{P}}

\def\Aut{\operatorname{Aut}}

\def\Coker{\operatorname{Coker}}

\def\dim{\operatorname{dim}}

\def\Ext{\operatorname{Ext}}

\def\gldim{\operatorname{gldim}}

\def\Hom{\operatorname{Hom}}
\def\id{\operatorname{id}}

\def\Im{\operatorname{Im}}
\def\injdim{\operatorname{injdim}}

\def\Ker{\operatorname{Ker}}

\def\Proj{\operatorname{Proj}}

\def\mod{\operatorname{\mathsf{mod}}}
\def\Mod{\operatorname{\mathsf{Mod}}}
\def\grmod{\operatorname{\mathsf{grmod}}}
\def\GrMod{\operatorname{\mathsf{GrMod}}}
\def\BiGrMod{\operatorname{\mathsf{BiGrMod}}}

\def\QBiGr{\operatorname{\mathsf{QBiGr}}}

\def\Tors{\operatorname{\mathsf{Tors}}}
\def\tors{\operatorname{\mathsf{tors}}}

\def\CM{\operatorname{\mathsf{CM}}}
\def\QGr{\operatorname{\mathsf{QGr}}}
\def\qgr{\operatorname{\mathsf{qgr}}}
\def\QMod{\operatorname{\mathsf{QMod}}}
\def\qmod{\operatorname{\mathsf{qmod}}}

\def\uCM{\operatorname{\underline{\mathsf{CM}}}}

\def\LTors{\operatorname{\mathsf{LTors}}}
\def\lgrmod{\operatorname{\mathsf{lgrmod}}}
\def\lltors{\operatorname{\mathsf{lltors}}}
\def\qlgr{\operatorname{\mathsf{qlgr}}}

\def\BiLTors{\operatorname{\mathsf{BiLTors}}}
\def\lbigrmod{\operatorname{\mathsf{lbigrmod}}}
\def\lbiltors{\operatorname{\mathsf{lbiltors}}}
\def\qlbigr{\operatorname{\mathsf{qlbigr}}}

\def\Db{\mathsf{D^b}}

\def\L{\operatorname{\mathsf{L}}}
\def\l{\operatorname{\mathsf{l}}}

\def\<{\langle}
\def\>{\rangle}

\numberwithin{equation}{section}

\begin{document}

\title[Twisted Segre products]
{Twisted Segre products}

\author{Ji-Wei He}
\address{School of Mathematics,
Hangzhou Normal University,
Hangzhou, Zhejiang 311121, China}
\email{jwhe@hznu.edu.cn}

\author{Kenta Ueyama}
\address{Department of Mathematics, Faculty of Education,
Hirosaki University,
1 Bunkyocho, Hirosaki, Aomori 036-8560, Japan}
\email{k-ueyama@hirosaki-u.ac.jp}


%16W50: Graded rings and modules (associative rings and algebras)
%14A22: Noncommutative algebraic geometry
%16S38: Rings arising from noncommutative algebraic geometry
%16E65: Homological conditions on associative rings (generalizations of regular, Gorenstein, Cohen-Macaulay rings, etc.)
%16G50 Cohen-Macaulay modules in associative algebras

\begin{abstract}
We introduce the notion of the twisted Segre product $A\circ_\psi B$ of $\mathbb Z$-graded algebras $A$ and $B$ with respect to a twisting map $\psi$.
It is proved that if $A$ and $B$ are noetherian Koszul Artin-Schelter regular algebras and $\psi$ is a twisting map such that the twisted Segre product $A\circ_\psi B$ is noetherian, then $A\circ_\psi B$ is a noncommutative graded isolated singularity. To prove this result, the notion of densely (bi-)graded algebras is introduced.
Moreover, we show that the twisted Segre product $A\circ_\psi B$ of $A=k[u,v]$ and $B=k[x,y]$ with respect to a diagonal twisting map $\psi$ is a noncommutative quadric surface (so in particular it is noetherian), and we compute the stable category of graded maximal Cohen-Macaulay modules over it.
\end{abstract}

\maketitle

%\tableofcontents

\section{Introduction}
Let $k$ be an algebraically closed field of characteristic $0$, and let $\PP^{d}$ be the $d$-dimensional projective space over $k$.

On the geometric side, the \emph{Segre embedding} is defined as the map
\begin{align*}
\varphi:\PP^{n-1} \times \PP^{m-1} &\to \PP^{nm-1}\\
((a_1, a_2,\dots, a_n), (b_1, b_2\dots, b_m)) &\mapsto (a_1b_1, a_2b_1,\dots, a_{n-1}b_m, a_nb_m).
\end{align*}
The map $\varphi$ is injective, and
the image of $\varphi$ is a subvariety of $\PP^{nm-1}$.
The \emph{Segre product} $X \circ Y$ of two projective
varieties $X \subseteq \PP^{n-1}$ and $Y \subseteq \PP^{m-1}$ is defined by the image of $X \times Y$ in $\PP^{nm-1}$ under the Segre embedding $\varphi$.

On the algebraic side, the \emph{Segre product} $A \circ B$ of two (not necessarily commutative) $\mathbb Z$-graded $k$-algebras $A=\bigoplus_{i \in \ZZ}A_i$ and $B=\bigoplus_{i \in \ZZ}B_i$ is defined to be the $\mathbb Z$-graded $k$-algebra $A \circ B := \bigoplus_{i \in \ZZ} (A_i \otimes_k B_i)$.

The algebraic side and the geometric side are related as follows:
if $X \subseteq \PP^{n-1}$ and $Y \subseteq \PP^{m-1}$ are projective varieties with the homogeneous coordinate rings $A$ and $B$, respectively, then the Segre product $A\circ B$ is the homogeneous coordinate ring for the Segre product $X \circ Y \subseteq \PP^{nm-1}$.
It is well-known that Segre products are essential in both algebraic geometry and commutative ring theory (see \cite{Har, GW}).

Let us consider the simplest case.
Let $A=k[u,v]$ and $B=k[x,y]$ be graded polynomial rings in two variables of degree one,
that is, copies of the homogeneous coordinate ring of the projective line $\PP^1$.
The Segre embedding
$\PP^1 \times \PP^1 \to \PP^3; ((a_1, a_2), (b_1, b_2)) \mapsto (a_1b_1, a_2b_1, a_1b_2, a_2b_2)$
embeds $\PP^1 \times \PP^1$ as a smooth quadric surface
$Q=V(XW-YZ)$ in $\PP^3= \Proj k[X,Y,Z,W]$, and
the Segre product $A \circ B = k[X, Y, Z, W]/(XW-YZ)$ is the homogeneous coordinate ring of $Q$, where $X= u\otimes x, Y=v\otimes x, Z=u\otimes y, W=v\otimes y$.

The purpose of this work is to generalize the notion of the Segre product of graded algebras for the development of noncommutative algebraic geometry.

Again, let us consider the simplest case.
Since (right and left) noetherian Koszul Artin-Schelter regular (abbreviated AS-regular) algebras are considered as nice noncommutative generalizations of polynomial algebras
in noncommutative algebraic geometry,
a natural noncommutative generalization of the Segre product $k[u,v] \circ k[x,y]$ is to replace $k[u,v]$ and $k[x,y]$ by $2$-dimensional noetherian Koszul AS-regular algebras.
However, this is not interesting in the following sense.
\begin{prop}[{\cite[Lemma 2.12]{VR}}] \label{prop.spas}
If $S$ and $T$ are $2$-dimensional noetherian Koszul AS-regular algebras, then the category of graded right modules over $S \circ T$
is equivalent to the category of graded right modules over $k[x,y] \circ k[u,v]$.
\end{prop}

In order to obtain a proper noncommutative generalization
(up to graded module category equivalence),
in this paper, we introduce and discuss the notion of \emph{twisted Segre product}.

Let $A$ and $B$ be $\mathbb Z$-graded algebras. We define a \emph{twisting map} $\psi:B\otimes A\to A\otimes B$ in Section \ref{sect-s}. A twisting map is indeed a smash product structure as introduced in \cite{CMZ} with more restrictions. Given a twisting map $\psi:B\otimes A\to A\otimes B$, we introduce the \emph{twisted Segre product} $A\circ_\psi B$ in Section \ref{sect-s}. The relations between the smash product $A\#_\psi B$ and $A\circ_\psi B$ is discussed in Section \ref{sect-d}.

The main result of the paper is the following theorem.

\begin{thm}[{Theorem \ref{thm-d4}}]\label{thm.imain}
Let $A$ and $B$ be noetherian Koszul AS-regular algebras, and let $\psi:B\otimes A \to A\otimes B$ be a twisting map.
Assume that the twisted Segre product $A\circ_\psi B$ is noetherian.
Then $A\circ_\psi B$ is a graded isolated singularity; that is,
the category $\qgr A\circ_\psi B$ has finite global dimension.
\end{thm}

Since $\qgr A\circ_\psi B$ is considered to be the category of coherent sheaves on the (imaginary) noncommutative projective scheme associated to $A\circ_\psi B$, this theorem can be regarded as a noncommutative version of the fact that $\PP^{n-1} \times \PP^{m-1}$ is smooth.
Noncommutative graded isolated singularities play an important role in Cohen-Macaulay representation theory (see \cite{Ue, SV, MUs, MU, HY}).

To prove Theorem \ref{thm.imain}, we introduce the concept of \emph{densely graded algebras} for $\mathbb Z$-graded algebras and for $\mathbb Z \times \mathbb Z$-bigraded algebras in Section \ref{sect-d}. Then we prove a version of Dade's Theorem for densely graded algebras (see, Theorems \ref{thm-d1} and \ref{thm-d2}). Applying a version of Dade's Theorem to the bigraded smash products $A\#_\psi B$ induces an equivalence between quotient categories of $A\#_\psi B$ and $A\circ_\psi B$ (see, Theorem \ref{thm-d3}).

Furthermore, in light of the usual Segre products of commutative rings, the following questions naturally arise.
\begin{ques}
Let $A$ and $B$ be noetherian Koszul AS-regular algebras of dimension $d\geq 2$ and $d'\geq 2$, respectively.
Let $\psi:B\otimes A\to A\otimes B$ be a twisting map.
\begin{enumerate}
\item Is $A\circ_\psi B$ noetherian?
\item Is $A\circ_\psi B$ Koszul?
%\item Is $A\circ_\psi B$ a domain?
\item  If $d=d'$, then is $A\circ_\psi B$ AS-Gorenstein?
\item If $d=d'=2$, then is $A\circ_\psi B$ a noncommutative quadric surface? That is, are there a noetherian Koszul AS-regular algebra $S$ of dimension $4$ and a regular normal element $f \in S_2$ such that
$A\circ_\psi B \cong S/(f)$?
\end{enumerate}
\end{ques}

In general, twisted Segre products of two noetherian Koszul AS-regular algebras are complicated, and we cannot give an answer to the above questions at present. However, we can obtain the following result as a partial answer.

\begin{thm}[{Theorem \ref{thm.rel}\textrm{(ii)}}] \label{thm.imain2}
Let $A=k[u,v], B=k[x,y]$ and let $\psi:B\otimes  A\to A\otimes B$ be a twisting map. Assume that $\psi$ is diagonal.
Then $A\circ_{\psi} B$ is a noncommutative quadric surface.
In particular, it is noetherian, Koszul, and AS-Gorenstein.
\end{thm}

Unlike the situation in Proposition \ref{prop.spas}, we can see that $A\circ_{\psi} B$ in the above theorem is a proper extension of $A\circ B$ up to equivalence of their graded module categories (see, Example \ref{ex.nqs}).

The stable category of graded maximal Cohen-Macaulay modules over a noncommutative quadric hypersurface (especially when it is an isolated singularity) is an interesting subject to study (see \cite{SV, HY, MU, HU, HMY}).
In the setting of Theorem \ref{thm.imain2}, $A\circ_{\psi}B$ is a noncommutative quadric surface and an isolated singularity by Theorems \ref{thm.imain} and \ref{thm.imain2},
so we calculate the stable category of graded maximal Cohen-Macaulay modules over $A\circ_{\psi}B$, denoted by $\uCM^{\mathbb Z}(A\circ_{\psi}B)$.

\begin{thm} [{Corollary \ref{cor-sing}}]
Let $A$, $B$, and $\psi$ be as in Theorem \ref{thm.imain2}.
Then there is an equivalence of triangulated categories \[\uCM^{\mathbb Z} (A\circ_\psi B)\cong \Db(\mod k \times k),\]
where $\mod k \times k$ is the category of finite dimensional right modules over the semisimple algebra $k \times k$, and $\Db(\mod k \times k)$ the bounded derived category of $\mod k \times k$.
\end{thm}

% % % % % % % % %
\section{Preliminaries}
% % % % % % % % %
Throughout, $k$ is an algebraically closed field of characteristic $0$. All algebras and vector spaces considered in this paper are over $k$, and all unadorned tensor products $\otimes$ are taken over $k$. An algebra is called \emph{noetherian} if it is right and left noetherian.

For an algebra $A$, we write $\Mod A$ for the category of right $A$-modules,
and $\mod A$ for the full subcategory of $\Mod A$ consisting of finitely generated modules.
For a $\mathbb Z$-graded algebra $A$, we write $\GrMod A$ for the category of graded right $A$-modules with degree zero $A$-module homomorphisms,
and $\grmod A$ for the full subcategory of $\GrMod A$ consisting of finitely generated modules.

A $\mathbb Z$-graded algebra $A=\bigoplus_{i\in\mathbb Z}A_i$ is called a \emph{connected} graded algebra if $A_i=0$ for all $i<0$ and $A_0=k$. A \emph{quadratic algebra} is a connected graded algebra which is generated in degree $1$ with quadratic relations; that is, $A\cong T(V)/(R)$, where $V$ is a finite dimensional vector space, $T(V)$ is the tensor algebra generated by $V$, and $R\subseteq V\otimes V$.
The \emph{quadratic dual} of a quadratic algebra $A=T(V)/(R)$ is defined to be $A^!=T(V^*)/(R^\bot)$, where $V^*$ is the dual vector space of $V$ and $R^\bot\subseteq V^*\otimes V^*$ is the orthogonal complement of $R$.
A connected graded algebra $A$ is called a \emph{Koszul algebra} if the trivial graded module $k \in \GrMod A$ has a linear resolution
$$\cdots\longrightarrow F^{-n}\longrightarrow\cdots \longrightarrow F^{-1}\longrightarrow F^0\longrightarrow k\longrightarrow0,$$
where $F^{-n}$ is a graded free module generated in degree $n$.
It is known that a Koszul algebra $A$ must be a quadratic algebra
and its quadratic dual $A^!$ is also a Koszul algebra (see \cite{Pr}).

Let $A$ be a $\mathbb Z$-graded algebra. For $M \in \GrMod A$ and $i \in \mathbb Z$, we define the \emph{shift} of $M$ by $i$ to be the graded right $A$-module $M(i)=\bigoplus_{j \in \mathbb Z}M(i)_j$, where $M(i)_j = M_{i+j}$.
For $M, N \in \GrMod A$, we define
\[ \Ext_A^j(M,N) = \bigoplus _{i\in \mathbb Z}\Ext^j_{\GrMod A}(M, N(i)).\]

A noetherian connected graded algebra $A$ is called an \emph{Artin-Schelter Gorenstein algebra} of dimension $d$ if the following conditions are satisfied (see \cite{AS}):
\begin{enumerate}
  \item  the injective dimensions $\injdim {}_AA=\injdim A_A=d<\infty$;
  \item $\Ext^i_A(k,A)=0$ for $i\neq d$ and $\Ext_A^d(k,A)\cong k(\ell)$ for some $\ell \in \mathbb Z$;
  \item  the left version of ${\rm(ii)}$ is satisfied.
\end{enumerate}
If further, $\gldim A=d$, then $A$ is called an \emph{Artin-Schelter regular algebra}. Below, we  will simply write ``AS'' for ``Artin-Schelter''.

\begin{rem}
The standard definition of Artin-Schelter regularity has the hypothesis of finite Gelfand-Kirillov (GK) dimension. However it is unnecessary here, since it follows from the noetherian property by \cite[Theorem 2.4]{SZ}.
\end{rem}

Let $A$ be a noetherian AS-Gorenstein algebra.
A finitely generated graded module $M\in \grmod A$ is called \emph{maximal Cohen-Macaulay} if $\Ext^i_A(M, A)=0$ for all $i \geq  1$.
The full subcategory of $\grmod A$ consisting of graded maximal Cohen-Macaulay modules is denoted by $\CM^{\ZZ}(A)$. The category $\CM^{\ZZ}(A)$ is a Frobenius category.
The \emph{stable category} of graded maximal Cohen-Macaulay modules, denoted by $\uCM^{\ZZ}(A)$, has the same objects as $\CM^{\ZZ}(A)$,
and the morphism space is given by
\[ \Hom_{\uCM^{\ZZ}(A)}(M, N) = \Hom_{\CM^{\ZZ}(A)}(M,N)/P(M,N), \]
where $P(M,N)$ consists of degree zero $A$-module homomorphisms factoring through a graded projective module. By \cite{Hap}, $\uCM^{\ZZ}(A)$ canonically has a structure of triangulated category.

Let $S$ be a noetherian Koszul AS-regular algebra of dimension $d$, and let $f\in S_2$ be a regular normal element of degree $2$. Then we call the quotient algebra $A=S/(f)$ a \emph{noncommutative quadric hypersurface}.
It is well known that a noncommutative quadric hypersurface $A=S/(f)$ is a noetherian Koszul AS-Gorenstein algebra of dimension $d-1$.
In particular, if $d=4$, then $A=S/(f)$ is called a \emph{noncommutative quadric surface}.

Let $S$ be a $d$-dimensional noetherian Koszul AS-regular algebra, $f\in S_2$ a regular normal element of degree $2$, and $A=S/(f)$.
Then there exists a unique regular normal element $w \in A_2^!$ up to scalar such that $A^!/(w) = S^!$ by \cite[Corollary 1.4]{ST}.
Since $w$ is regular normal, there exists a unique graded algebra automorphism $\nu_w$ of $A^!$ such that $aw=w\nu_w(a)$ for $a\in A^!$. We call $\nu_w$ the \emph{normalizing automorphism} of $w$.
Then we define
\[ C(A) = A^![w^{-1}]_0, \]
where $A^![w^{-1}]:=\{aw^{-i}\mid a\in A^!, i\in \NN\}$ is a $\mathbb Z$-graded algebra with
\begin{itemize}
\item (addition)\  $aw^{-i}+bw^{-j}=(aw^j+bw^i)w^{-i-j}$;
\item (multiplication)\ $(aw^{-i})(bw^{-j})=a\nu_w^i(b)w^{-i-j}$;
\item (grading)\ $\deg(aw^{-i})=\deg a-2i$.
\end{itemize}

For a noncommutative quadric hypersurface $A=S/(f)$, it is known that $C(A)$ is essential to compute the stable category $\uCM^{\ZZ}(A)$.

\begin{prop}[{\cite[Lemma 4.13]{MU}; see also \cite[Lemma 5.1, Proposition 5.2]{SV} for the case where $f$ is central}]
\label{prop.C(A)}
Let $S$ be a $d$-dimensional noetherian Koszul AS-regular algebra, $f\in S_2$ a regular normal element of degree $2$, and $A=S/(f)$.
Then
%\begin{enumerate}
%\item $\dim_k C(A)=2^{d-1}$.
%\item
there exists an equivalence of triangulated categories
\[\uCM^{\ZZ}(A) \overset{\cong}{\longrightarrow} \Db(\mod C(A)),\]
where $\Db(\mod C(A))$ is the bounded derived category of $\mod C(A)$.
%\end{enumerate}
\end{prop}

Let $A$ be a noetherian $\mathbb Z$-graded algebra.
A homogeneous element $x$ of a graded module $M \in \GrMod A$ is called a \emph{torsion element} if $\dim_k xA<\infty$.
A graded module $M \in \GrMod A$ is called a {\it torsion module} if every homogeneous element of $M$ is a torsion element.
Write $\Tors A$ for the full subcategory of $\GrMod A$ consisting of all torsion modules.
It is well-known that $\Tors A$ is a Serre subcategory of $\GrMod A$.
Hence we have the quotient category
\[ \QGr A:=\GrMod A/\Tors A.\]
We write $\tors A:= \Tors A \cap \grmod A$ and $\qgr A:=\grmod A/\tors A$.
The category $\QGr A$ (resp. $\qgr A$) is considered as the category of quasi-coherent sheaves (resp. coherent sheaves) on the (imaginary) noncommutative projective scheme associated to $A$ (see \cite{AZ}).

\begin{dfn}[{\cite[Definition 2.2]{Ue}}]
A noetherian $\mathbb Z$-graded algebra is a \emph{noncommutative graded isolated singularity} if $\qgr A$ has finite global dimension.
\end{dfn}

For a noncommutative quadric hypersurface $A=S/(f)$, the following characterization of when $A$ is a noncommutative graded isolated singularity is known.

\begin{prop}[{\cite[Theorem 5.5]{MU}; see also \cite[Theorem 6.3]{HY} for the case where $f$ is central}]
\label{prop.isol}
Let $S$ be a $d$-dimensional noetherian Koszul AS-regular algebra, $f\in S_2$ a regular normal element of degree $2$, and $A=S/(f)$.
Then $A$ is a noncommutative graded isolated singularity if and only if $C(A)$ is a semisimple
algebra.
\end{prop}

% % % % % % % % % % % % % % % % % %
\section{Twisted Segre products}\label{sect-s}
% % % % % % % % % % % % % % % % % %

In this section, we give the definition of a twisted Segre product.
We first recall some conventional notations from \cite{CMZ}. Let $V$ and $U$ be vector spaces, and let $\psi:U\otimes V\to V\otimes U$ be a linear map. For $u\in U,v\in V$, we write
\[\psi(u\otimes v)=v_\psi\otimes u^\psi=v_\Psi\otimes u^\Psi,\]
where the summation is understood, so we omit the summation symbols.

Let $M=\bigoplus_{i\in \mathbb Z} M_i$ and $N=\bigoplus_{i\in\mathbb Z}N_i$ be $\mathbb Z$-graded vector spaces. Then we define a $\mathbb Z$-graded vector space $M\circ N$ by
\[M\circ N=\bigoplus_{i\in\mathbb Z}(M_i\otimes N_i).\]

Let $A$ and $B$ be $\mathbb Z$-graded algebras. We call a bijective linear map $\psi:B\otimes A\to A\otimes B$ a {\it twisting map} if the following conditions are satisfied:
\begin{eqnarray}
\label{tm1}  &\psi(B_i\otimes A_j)\subseteq A_j\otimes B_i  & \lefteqn{\forall i,j\in\mathbb Z;}\\
 \label{tm2} &\psi(1\otimes a)=a\otimes 1 &\lefteqn{\forall a\in A;}\\
 \label{tm3}  &\psi(b\otimes 1)=1\otimes b &\lefteqn{\forall b\in B;}\\
\label{tm4}   &\psi(bc\otimes a)=a_{\psi\Psi}\otimes b^\Psi c^\psi &\lefteqn{\forall b,c\in B, \forall a\in A;}\\
\label{tm5}   &\psi(b\otimes ad)=a_{\psi}d_\Psi\otimes b^{\psi\Psi} &\lefteqn{\forall b\in B, \forall a,d\in A.}
\end{eqnarray}

\begin{rem}\label{remnew}
(1) The conditions (\ref{tm2})--(\ref{tm5}) have appeared in \cite{CIMZ} and \cite{CMZ}.

(2) The conditions (\ref{tm4}) and (\ref{tm5}) say that we have the following commutative diagrams
\begin{align*}
\xymatrix@C=3pc@R=2pc{
B \otimes B \otimes A \ar[r]^{\id_{B}\otimes \psi} \ar[d]^{m_{B}\otimes \id_A}  & B \otimes A \otimes B \ar[r]^{\psi \otimes \id_{B}}  & A \otimes B \otimes B \ar[d]^{\id_A\otimes m_{B}}\\
B \otimes A \ar[rr]^{\psi}  &&A \otimes B,
}
\end{align*}
\begin{align*}
\xymatrix@C=3pc@R=2pc{
B \otimes A \otimes A \ar[r]^{\psi\otimes \id_{A}} \ar[d]^{\id_{B}\otimes m_{A}}  & A \otimes B \otimes A \ar[r]^{\id_{A}\otimes \psi}  & A \otimes A \otimes B \ar[d]^{m_{A}\otimes \id_B}\\
B \otimes A \ar[rr]^{\psi}  &&A \otimes B,
}
\end{align*}
where $m_A$ and $m_B$ are the multiplications of $A$ and $B$, respectively.
\end{rem}

With a twisting map $\psi: B\otimes A\to A\otimes B$, one may define a product on $A\circ B$ as follows: for $n,m\in\mathbb Z$, $a\in A_n,c\in A_m$, and $b\in B_n,d\in B_m$,
\begin{equation}\label{tsp}
  (a\otimes b)(c\otimes d)=ac_\psi\otimes b^\psi d.
\end{equation}
Then we have the following proposition, whose proof is based on \cite[Theorem 2.5]{CIMZ} (see also \cite[Theorem 7 in Chapter 2]{CMZ}).

\begin{prop}\label{prop-tsp} Let $A$ and $B$ be $\mathbb Z$-graded algebras, and let $\psi:B\otimes A\to A\otimes B$ be a twisting map. With the product defined by (\ref{tsp}), $A\circ B$ is a $\mathbb Z$-graded algebra.
\end{prop}
\begin{proof} As mentioned in Remark \ref{remnew}(1), the conditions (\ref{tm2})--(\ref{tm5}) are equivalent to the conditions in \cite[Theorem 2.5(3)]{CIMZ}. Hence, by \cite[Theorem 2.5]{CIMZ}, $A\circ B$ is an associative algebra. Note that $A\circ B$ is a graded vector space with homogeneous component $(A\circ B)_n=A_n\otimes B_n$ for $n\in \mathbb Z$. The condition (\ref{tm1}) ensures that the product defined by (\ref{tsp}) preserves the grading of $A\circ B$.
\end{proof}

Henceforth, we write
\[A\circ_\psi B\]
for the graded algebra obtained in the above proposition, and we call $A\circ_\psi B$ the \emph{twisted Segre product} of $A$ and $B$ with respect to $\psi$. If $\psi$ is the flip map, i.e., $\psi(b\otimes a)=a\otimes b$ for all $a\in A$ and $b\in B$, then $A\circ_\psi B$ is the usual Segre product of $A$ and $B$, so we simply write $A\circ B$ instead of $A\circ_\psi B$.

\begin{rem}\label{rem-tsp}
(1) Let $A$ and $B$ be connected graded algebras which are generated in degree $1$. Assume $\psi:B\otimes A\to A\otimes B$ is a twisting map. By Conditions (\ref{tm1}), (\ref{tm4}), and (\ref{tm5}), one sees that $\psi$ is determined by its restriction to $B_1\otimes A_1$.

(2) Let $V$ and $U$ be finite dimensional vector spaces, and let $\psi_0:U\otimes V\to V\otimes U$ be a bijective linear map. By Conditions (\ref{tm4}) and (\ref{tm5}), we see that $\psi_0$ induces uniquely a twisting map
\[\psi_T:T(U)\otimes T(V)\longrightarrow T(V)\otimes T(U).\]
Hence we have the twisted Segre product $T(V)\circ_{\psi_T} T(U)$ with respect to $\psi_T$.
We can define a unique algebra homomorphism
\begin{equation}\label{eq-tenalg1}
  \Xi:T(V\otimes U)\longrightarrow T(V)\circ_{\psi_T} T(U)
\end{equation}
such that $\Xi(v\otimes u)=v\otimes u$ for $v\in V,u\in U$.
Since $\psi_{T}$ is a bijective map, it follows that $\Xi$ is indeed an isomorphism of graded algebras. Therefore, we obtain
\begin{equation}\label{eq-tenalg2}
  T(V\otimes U)\cong T(V)\circ_{\psi_T} T(U)\cong  T(V)\circ T(U).
\end{equation}

(3) Let $A$ and $B$ be graded algebras, and let $\psi:B\otimes A\to A\otimes B$ be a twisting map. By \cite[Theorem 7 in Chapter 2]{CMZ}, we obtain a \emph{smash product} of $A$ and $B$, denoted by $A\#_\psi B$, whose underlying vector space is equal to $A\otimes B$ and whose multiplication is defined by
$(a\otimes b)(c\otimes d)=ac_\psi\otimes b^\psi d$, for $a,c\in A$ and $b,d\in B$. Then we see that $A\circ_\psi B$ is an (ungraded) subalgebra of $A\#_\psi B$.
Note that $A\#_\psi B$ is also called a \emph{twisted tensor product} of $A$ and $B$ in some literature (\cite{CSV, CG, JPS, WS, SZL} for example).
\end{rem}

The next lemma is straightforward.

\begin{lem}\label{lem-tspmorp} Let $\psi:B\otimes A\to A\otimes B$ and $\psi':B'\otimes A'\to A'\otimes B$ be two twisting maps. Assume $f:A\to A'$ and $g:B\to B'$ are algebra homomorphisms such that the following diagram commutes
$$\xymatrix{
  B\otimes A \ar[d]_{g\otimes f} \ar[r]^{\psi} & A\otimes B \ar[d]^{f\otimes g} \\
  B'\otimes A' \ar[r]^{\psi'} & A'\otimes B'.   }$$
Then $h_{fg}: A\circ_\psi B\to A'\circ_{\psi'}B'$ defined by $h_{fg}(a \otimes b)=f(a) \otimes g(b)$ is an algebra homomorphism.
\end{lem}

 Let $V$ and $U$ be finite dimensional vector spaces.
 Fix a basis $x_1,\dots, x_n$ for $V$ and a basis $y_1,\dots, y_m$ for $U$ so that $T(V) = k\<x_1,\dots, x_n\>$ and $T(U) = k\<y_1,\dots, y_m\>$. Let $A=k\<x_1,\dots, x_n\>/I_A$ and $B=k\<y_1,\dots, y_m\>/I_B$ be connected graded algebras generated in degree $1$, and let $\psi:B\otimes A\to A\otimes B$ be a twisting map.
 Then we have
\begin{align*}
\psi( y_i \otimes a) = \sum_{j=1}^{m} a_{ij} \otimes y_j
\end{align*}
for some $a_{ij} \in A$.
By Condition (\ref{tm5}) and the linear independence of $y_1,\dots, y_m$, one can check that
\begin{equation}\label{eq-sigma}
\sigma_\psi = (\sigma_{ij}) : A \longrightarrow M_m(A)
\end{equation}
defined by $\sigma_{ij}(a)=a_{ij}$ is an algebra homomorphism.
Moreover, since $\psi$ is bijective, $\sigma_\psi$  is \emph{invertible} in the following sense (see \cite[\S 5.5]{JPS}, \cite[Definition 4]{WS}): there is an algebra homomorphism $\tau=(\tau_{ij}): A \to M_m(A)$ such that
\[
\sum_{k=1}^{m}\tau_{ki}(\sigma_{jk}(a))=\sum_{k=1}^{m}\sigma_{ki}(\tau_{jk}(a))= \delta_{ij}a
\]
for any $a \in A$, where $\delta_{ij}$ denotes the Kronecker delta.
The invertibility of $\sigma_\psi$ plays an important role in the study of smash products $A\#_\psi B$ (see \cite{JPS, WS, SZL} for example).

In Section \ref{sec-2poly}, we will focus on the following special case.

\begin{dfn}\label{def-diag}
Let $A=k\<x_1,\dots, x_n\>/I_A$ and $B=k\<y_1,\dots, y_m\>/I_B$ be connected graded algebras generated in degree $1$. Let $\psi:B\otimes A\to A\otimes B$ be a twisting map, and $\sigma_\psi=(\sigma_{ij}): A \to M_m(A)$ the algebra homomorphism given in
(\ref{eq-sigma}).
We say that $\psi$ is \emph{diagonal} if $\sigma_{ij}=0$ for $i \neq j$.
% % %Moreover, if $\psi$ is diagonal, then we say that the twisted Segre product $A\circ_\psi B$ is \emph{diagonal}.
\end{dfn}

The flip map $\psi:B\otimes A\to A\otimes B$ is clearly an example of a diagonal twisting map.

% % % % % % % % % % % % % % % % % % % % % % %
\section{Quotient categories of twisted Segre products}\label{sect-d}
% % % % % % % % % % % % % % % % % % % % % % %

In this section, we will prove that if a twisted Segre product of two noetherian Koszul AS-regular algebras is
noetherian, then it is a noncommutative graded isolated singularity. We need some preparations.

\subsection{Densely graded algebras}

This subsection is devoted to giving a version of Dade's Theorem for densely $\mathbb Z$-graded algebras.

\begin{dfn}
Let $S=\bigoplus_{n\in \mathbb Z}S_n$ be a $\mathbb Z$-graded algebra. We say that $S$ is a \emph{densely graded algebra} if $\dim_k(S_n/S_iS_{n-i})<\infty$ for all $n$ and $i$.
\end{dfn}
The concept of densely graded algebras was firstly introduced in \cite{HVOZ} with more restrictions. Densely graded algebras are generalizations of strongly graded algebras. We have the following properties.

\begin{lem}\label{lem-d1}
Let $S$ be a $\mathbb Z$-graded algebra. Assume that $S_0$ is a noetherian algebra and $S_i$ is finitely generated both as a left $S_0$-module and a right $S_0$-module for every $i\in\mathbb Z$.
Then we have the following properties:
 \begin{enumerate}
   \item  $S$ is a densely graded algebra if and only if $\dim_k(S_0/S_iS_{-i})<\infty$ for all $i$.
   \item Assume that $S$ is a densely graded algebra, and $M=\bigoplus_{n\in\mathbb Z}M_n$ is a graded right $S$-module such that $M_n$ is a finitely generated $S_0$-module for every $n$. Then $M_n$ is finite dimensional over $k$ for every $n$ if and only if $M_0$ is finite dimensional over $k$.
 \end{enumerate}
\end{lem}

\begin{proof}
(i) Assume $\dim_k(S_0/S_iS_{-i})<\infty$ for all $i$. For every $i\in\mathbb Z$, we have an exact sequence
\begin{equation}\label{eq-d2}
  0\longrightarrow S_iS_{-i}\longrightarrow S_0\longrightarrow S_0/S_iS_{-i}\longrightarrow 0.
\end{equation}
Applying the functor $S_n\otimes_{S_0}-$ to the above sequence, we obtain the following exact sequence
\begin{equation}\label{eq-d1}
  S_n\otimes_{S_0} (S_iS_{-i})\longrightarrow S_n\longrightarrow S_n\otimes_{S_0}(S_0/S_iS_{-i})\longrightarrow 0.
\end{equation}
Since $S_n$ is finitely generated as a right $S_0$-module and $S_0/S_iS_{-i}$ is finite dimensional, it follows that $S_n\otimes_{S_0}(S_0/S_iS_{-i})$ is finite dimensional. Note that the image of the left morphism in the exact sequence (\ref{eq-d1}) is equal to $S_n(S_iS_{-i})$. Hence $S_n/(S_n(S_iS_{-i}))$ is finite dimensional. Since $S_nS_i\subseteq S_{n+i}$, it follows $S_n(S_iS_{-i})\subseteq S_{n+i}S_{-i}$ and hence $S_n/(S_{n+i}S_{-i})$ is a quotient module of $S_n/(S_n(S_iS_{-i}))$. Therefore $S_n/(S_{n+i}S_{-i})$ is finite dimensional for every $i$.

(ii) Assume that $M_0$ is finite dimensional. Similar to the proof of (i), if we apply the functor $M_n\otimes_{S_0}-$ to the exact sequence
\[0\longrightarrow S_{-n}S_{n}\longrightarrow S_0\longrightarrow S_0/S_{-n}S_{n}\longrightarrow 0,\]
we obtain that $M_n/(M_nS_{-n}S_n)$ is finite dimensional. Since $M_nS_{-n}\subseteq M_0$, it follows that $M_nS_{-n}$ is finite dimensional. Since $S_n$ is a finitely generated $S_0$-module, it follows that $M_nS_{-n}S_n$ is finite dimensional. Hence $M_n$ is finite dimensional.
\end{proof}

\begin{cvn}\label{cvn-den}
Throughout the rest of this subsection, we always assume that $S$ is a densely graded algebra, $S_0$ is noetherian, and $S_i$ is finitely generated both as a left $S_0$-module and as a right $S_0$-module for every $i\in\mathbb Z$.
\end{cvn}

A homogeneous element $x$ of a graded module $M \in \GrMod S$ is called a \emph{locally torsion element} if the graded submodule $xS$ is locally finite; i.e., $\dim_k (xS)_n<\infty$ for all $n\in\mathbb Z$. A graded module $M \in \GrMod S$ is called a \emph{locally torsion module} if every homogeneous element of $M$ is a locally torsion element. Let $\LTors S$ be the full subcategory of $\GrMod S$ consisting of locally torsion modules. Since $S_0$ is noetherian and $S$ is densely graded, $\LTors S$ is a Serre subcategory of $\GrMod S$. Hence we have the quotient category
\[\QGr_{\L} S:=\GrMod S/\LTors S.\]
Let $\pi_{\L}:\GrMod S\to \QGr_{\L} S$ be the projection functor. It admits a right adjoint functor $\omega_{\L}:\QGr_{\L} S\to \GrMod S$.

We write $\lgrmod S$ for the full subcategory of $\GrMod S$ consisting of graded modules $M$ such that $M_i$ is a finitely generated $S_0$-module for all $i$, and let $\lltors S$ be the full subcategory of $\lgrmod S$ consisting of all locally torsion modules.
%Note that for an object $M \in \lgrmod S$, $M\in \lltors S$ if and only if $M_0$ is finite dimensional by Lemma \ref{lem-d1}(ii).
We also have the quotient category
\[\qlgr_{\l} S:=\lgrmod S/\lltors S.\]

An element $y$ of a module $N \in \Mod S_0$ is called a {\it torsion element} if $y(S_0)$ is finite dimensional. If every element of $N \in \Mod S_0$ is a torsion element, then $N$ is called a {\it torsion module}.
We write $\Tors S_0$ (resp. $\tors S_0$) for the full subcategory of $\Mod S_0$ (resp. $\mod S_0$) consisting of all torsion modules.
Note that, for a module $N \in \Mod S_0$, $N \in \tors S_0$ if and only if $N$ is finite dimensional.
We have the quotient categories
\[ \QMod S_0:=\Mod S_0/\Tors S_0 \quad \text{and} \quad
\qmod S_0:=\mod S_0/\tors S_0.\]
We refer to \cite[Chapter 4]{P} for detailed properties of the quotient categories. We use $\pi_0:\Mod S_0 \to \QMod S_0$ to denote the projection functor and $\omega_0:\QMod S_0 \to \Mod S_0$ to denote its right adjoint functor.

\begin{lem} \label{lem-d2}
Let $M$ be a graded right $S$-module. Then
the following are equivalent:
\begin{enumerate}
\item $M$ is a locally torsion $S$-module;
\item $M_i$ is a torsion $S_0$-module for every $i$;
\item For every homogeneous element $x \in M$, $xS_0$ is finite dimensional;
\item $M_0$ is a torsion $S_0$-module.
\end{enumerate}
\end{lem}

\begin{proof}
Since $xS_0=(xS)_i$ for every $x \in M_i$, we have $\textrm{(i)}\Rightarrow \textrm{(ii)}$.
Moreover, it is easy to see that $\textrm{(iii)} \Leftrightarrow \textrm{(ii)} \Rightarrow \textrm{(iv)}$ hold, so we only need to prove $\textrm{(iv)}\Rightarrow \textrm{(i)}$.
Assume $M_0$ is a torsion $S_0$-module. Let $x\in M_n$ be a homogeneous element. Since, by assumption, $S_{i}$ is a finitely generated $S_0$-module for all $i$, it follows that each homogeneous component of $xS$ is a finitely generated $S_0$-module. On the other hand, $(xS)_0$ is an $S_0$-submodule of $M_0$, so $(xS)_0$ is a torsion $S_0$-module. Since $(xS)_0$ is a finitely generated $S_0$-module, it follows that $(xS)_0$ is finite dimensional. By Lemma \ref{lem-d1}(ii), each homogeneous component of $xS$ is finite dimensional. Hence $x$ is a locally torsion element.
\end{proof}

By Lemma \ref{lem-d2},
we see that, for an object $M \in \lgrmod S$, $M\in \lltors S$ if and only if $M_0$ is finite dimensional.

\begin{cor}\label{cor-d1}
Let $M$ be a graded right $S$-module. Let $\mu:M_0\otimes_{S_0} S \to M$ be the graded right $S$-module morphism induced by the right $S$-action. Then both $\Ker(\mu)$ and $\Coker(\mu)$ are objects in $\LTors S$.
\end{cor}

\begin{proof}
Note that $(M_0\otimes_{S_0}S)_0\cong M_0$. Hence the restriction of $\mu$ to the degree zero component of $M_0\otimes_{S_0} S$ is an isomorphism. Therefore the degree zero component of $\Ker(\mu)$ and that of $\Coker(\mu)$ are equal to zero. By Lemma \ref{lem-d2}, both $\Ker(\mu)$ and $\Coker(\mu)$ are objects in $\LTors S$.
\end{proof}

\begin{cor}\label{cor-d2}
Let $f:X\to Y$ be a morphism in $\Mod S_0$ such that $\Ker(f), \Coker(f) \in \Tors S_0$. Then $\pi_{\L}(f\otimes_{S_0}S):\pi_{\L}(X\otimes_{S_0} S) \to \pi_{\L}(Y\otimes_{S_0}S)$ is an isomorphism in $\QGr_{\L} S$.
\end{cor}

\begin{proof} Let $\varphi=f\otimes_{S_0}S$. We only need to prove $\Ker(\varphi),\Coker(\varphi)\ in \LTors S$. Note that $(X\otimes_{S_0} S)_0\cong X$ and $(Y\otimes_{S_0} S)_0\cong Y$. Hence $\Ker(\varphi)_0\cong\Ker(f)$ and $\Coker(\varphi)_0\cong \Coker(f)$ in $\Mod S_0$. Therefore, $\Ker(\varphi), \Coker(\varphi)\in\LTors S$ by Lemma \ref{lem-d2}.
\end{proof}

\begin{lem}\label{lem-d3}
Let $N\in\Mod S_0$ be a torsion-free module; that is, $N$ contains no torsion elements. Then $\omega_0\pi_0(N)\cong \omega_{\L}\pi_{\L}(\omega_0\pi_0(N)\otimes_{S_0}S)_0$.
\end{lem}
\begin{proof} Let $E$ be the injective envelope of $N$. Then $\omega_0\pi_0(N)$ is isomorphic to the maximal submodule $M$ of $E$ which contains $N$ and satisfies that $M/N$ is a torsion module. We identify $\omega_0\pi_0(N)$ with $M$, and write $\iota:\omega_0\pi_0(N)\to E$ for the inclusion map. Then $E/\iota(\omega_0\pi_0(N))$ is torsion-free. Let $\nu:\omega_0\pi_0(N)\otimes_{S_0}S\to \omega_{\L}\pi_{\L}(\omega_0\pi_0(N)\otimes_{S_0}S)$ be the adjunction map. Then both $\Ker(\nu)$ and $\Coker(\nu)$ are objects in $\LTors S$
(see \cite[Proposition 4.3 in Chapter 4]{P}).
 By Lemma \ref{lem-d2}, $\Ker(\nu)_0$ and $\Coker(\nu)_0$ are torsion $S_0$-modules. Since $\Ker(\nu)_0$ is a submodule of $(\omega_0\pi_0(N)\otimes_{S_0}S)_0\cong \omega_0\pi_0(N)$ which is torsion-free, it follows that $\Ker(\nu)_0=0$. Let $\nu_0$ be the restriction of $\nu$ to the degree zero component $(\omega_0\pi_0(N)\otimes_{S_0}S)_0\cong \omega_0\pi_0(N)$. Then there is a right $S_0$-module morphism $\theta:\omega_{\L}\pi_{\L}(\omega_0\pi_0(N)\otimes_{S_0}S)_0\to E$ such that $\iota=\theta\circ\nu_0$. Since $\iota$ is injective, $\Ker(\theta)\cap\Im(\nu_0)=0$. Hence $\Ker(\theta)$ is isomorphic to a submodule of $\Coker(\nu_0)$. Since $\Coker(\nu_0)=\Coker(\nu)_0$ is a torsion module, it follows that $\Ker(\theta)$ is also a torsion module. On the other hand, $\Ker(\theta)$ is a submodule of the torsion-free module $\omega_{\L}\pi_{\L}(\omega_0\pi_0(N)\otimes_{S_0}S)_0$. It follows that $\Ker(\theta)=0$, and hence $\theta$ is an injective morphism. Therefore $\theta$ induces an injective morphism $\overline{\theta}:\Coker(\nu)_0\to E/\iota(\omega_0\pi_0(N))$. Since $E/\iota(\omega_0\pi_0(N))$ is torsion-free, it follows $\Coker(\nu)_0=0$. Therefore $\nu_0$ is indeed an isomorphism.
\end{proof}

Now we ready to state the following version of Dade's Theorem for densely graded algebras, which first appeared in \cite{HVOZ} with the assumption that $S$ is a $G$-graded algebra, where $G$ is a finite group.

\begin{thm}\label{thm-d1}
Let $S$ be as Convention \ref{cvn-den}.
Then the natural functor $(-)_0: \GrMod S \longrightarrow \Mod S_0; M\mapsto M_0$ induces the following equivalences of abelian categories:
\[(-)_0: \QGr_{\L} S\overset{\cong}{\longrightarrow} \QMod S_0
\quad \text{and}\quad
(-)_0: \qlgr_{\l} S\overset{\cong}{\longrightarrow} \qmod S_0.\]
\end{thm}

\begin{proof}
Recall that $\pi_0:\Mod S_0\longrightarrow \QMod S_0$ is the projection functor, and $\omega_0:\QMod S_0\longrightarrow\Mod S_0$ is its right adjoint functor.
Let $-\otimes_{S_0} S:\Mod S_0\longrightarrow \GrMod S$ be the canonical functor.
Let
\[F=(-\otimes_{S_0}S)\circ \omega_0: \QMod S_0\longrightarrow \GrMod S
 \quad \text{and} \quad
 G=\pi_0 \circ (-)_0: \GrMod S \longrightarrow \QMod S_0.\]
 Note that the functor $G$ is exact. By Lemma \ref{lem-d2}, for an object $M\in\GrMod S$, $G(M)=\pi_0(M_0)=0$ if and only if $M_0\in\Tors S_0$ if and only if $M\in \LTors S$. Therefore, by \cite[Corollary 3.11 in Chapter 4]{P}, $G$ induces an exact functor (also denoted by $(-)_0$ by abuse of notation)
 \[(-)_0:\QGr_{\L} S\longrightarrow\QMod S_0.\]
 It is easy to see that $GF$ is natural isomorphic to the identity functor $\id_{\QMod S_0}$. Hence $G$ is essentially surjective. Therefore $(-)_0:\QGr_{\L} S\longrightarrow\QMod S_0$ is essentially surjective.

Recall that $\pi_{\L}:\GrMod S\longrightarrow \QGr_{\L} S$ is the projection functor, and $\omega_{\L}$ is its right adjoint functor. We next prove that
\[ \Hom_{\QGr_{\L} S}(\pi_{\L}(M),\pi_{\L}(N))\cong \Hom_{\QMod S_0}(\pi_0(M_0),\pi_0(N_0))\]
for $M,N\in \GrMod S$.
Since $\omega_0$ is right adjoint to $\pi_0$, we have a natural isomorphism
\begin{equation}\label{eq-d3}
  \Hom_{\QMod S_0}(\pi_0(M_0),\pi_0(N_0))\cong\Hom_{\Mod S_0}(M_0,\omega_0\pi_0(N_0)).
\end{equation}
%We have the following natural isomorphism
%\begin{eqnarray}
%\Hom_{\GrMod S}(M_0\otimes_{S_0}S,\omega\pi(N_0)\otimes_{S_0}S)&\cong&\Hom_{\Mod S_0}(M_0,(\omega\pi(N_0)\otimes_{S_0}S)_0)\\
%\nonumber&\cong&\Hom_{\Mod S_0}(M_0,\omega\pi(N_0)).
%\end{eqnarray}
By Corollary \ref{cor-d1}, we have $\pi_{\L}(M)\cong \pi_{\L}(M_0\otimes_{S_0}S)$ and $\pi_{\L}(N)\cong \pi_{\L}(N_0\otimes_{S_0}S)$.
Let $\nu_{N_0}:N_0\to \omega_0\pi_0(N_0)$ be the adjunction morphism. Then both $\Ker(\nu_{N_0})$ and $\Coker(\nu_{N_0})$ are in $\Tors S_0$ (see \cite[Proposition 4.3 in Chapter 4]{P}). By Corollary \ref{cor-d2}, we obtain a natural isomorphism $\pi_{\L}(N_0\otimes_{S_0}S)\cong\pi_{\L}(\omega_0\pi_0(N_0)\otimes_{S_0}S)$.  Hence we have the following natural isomorphisms
\begin{eqnarray}
% \nonumber to remove numbering (before each equation)
  \Hom_{\QGr_{\L} S}(\pi_{\L}(M),\pi_{\L}(N))
  &\cong&\Hom_{\QGr_{\L} S}(\pi_{\L}(M_0\otimes_{S_0}S),\pi_{\L}(N_0\otimes_{S_0}S)) \\
  &\cong&\Hom_{\QGr_{\L} S}(\pi_{\L}(M_0\otimes_{S_0}S),\pi_{\L}(\omega_0\pi_0(N_0)\otimes_{S_0}S))\\
  &\cong&  \Hom_{\GrMod S}(M_0\otimes_{S_0}S,\omega_{\L}\pi_{\L}(\omega_0\pi_0(N_0)\otimes_{S_0}S))\\
  \label{eq-d4.1}
  &\cong&\Hom_{\Mod S_0}\left(M_0,\omega_{\L}\pi_{\L}(\omega_0\pi_0(N_0)\otimes_{S_0}S)_0\right)\\
 \label{eq-d4.2}
 &\cong&\Hom_{\Mod S_0}\left(M_0,\omega_0\pi_0(N_0)\right),
\end{eqnarray}
where
the isomorphism (\ref{eq-d4.1}) follows from the fact that $-\otimes_{S_0}S$ is left adjoint to the functor $(-)_0$, and
the isomorphism (\ref{eq-d4.2}) follows from Lemma \ref{lem-d3}. Combining the isomorphisms from (\ref{eq-d3}) to (\ref{eq-d4.2}), we obtain the desired natural isomorphism. In summarizing, we obtain that $(-)_0:\QGr_{\L} S\longrightarrow \QMod S_0$ is an equivalence.

Since $\LTors S \cap \lgrmod S= \lltors S$, it follows $\qlgr_{\l} S$ is a full subcategory of $\QGr_{\L} S$. It is clear that the functor $(-)_0:\QGr_{\L} S\longrightarrow\QMod S_0$ sends the objects of the subcategory $\qlgr_{\l} S$ to the objects of $\qmod S_0$, and for each object $\mathcal{N}$ in $\qmod S_0$, there is an object $\mathcal M$ in $\qlgr_{\l} S$ such that $(\mathcal{M})_0=\mathcal{N}$. Hence $(-)_0:\qlgr_{\l} S \longrightarrow\qmod S_0$ is an equivalence.
\end{proof}

\subsection{Dade's Theorem for densely bigraded algebras}

In this subsection, we consider bigraded algebras. Let $S=\bigoplus_{i,j\in\mathbb Z}S_{(i,j)}$ be a $\mathbb Z\times\mathbb Z$-bigraded algebra.
%For the sake of convenience, for an element $a\in S_{(i,j)}$, we call $i$ the {\it external degree} of $a$ and $j$ the {\it internal degree} %of $a$.
Write $S_i=\bigoplus_{j\in\mathbb Z}{S_{(i,j)}}$ for each $i\in\mathbb Z$. Then $S=\bigoplus_{i\in\mathbb Z}S_i$ can be viewed as a $\mathbb Z$-graded algebra.
We call $S$ a \emph{densely bigraded algebra} if $S$, viewed as a $\mathbb Z$-graded algebra, is a densely graded algebra.

\begin{cvn} \label{cvn-biden}
Throughout the rest of this subsection, we always assume that $S$ is a densely bigraded algebra such that $S_0:=\bigoplus_{j\in\mathbb Z}S_{(0,j)}$ is a noetherian algebra, and $S_i:=\bigoplus_{j\in\mathbb Z}S_{(i,j)}$ is finitely generated both as a graded left $S_0$-module and as a graded right $S_0$-module for every $i\in \mathbb Z$.
\end{cvn}

Denote by $\BiGrMod S$  the category of all bigraded right $S$-modules, and by $\lbigrmod S$ the full subcategory of $\BiGrMod S$ consisting of all bigraded right $S$-modules $M=\bigoplus_{i,j \in \mathbb Z}M_{(i,j)}$ such that, for every $i\in\mathbb Z$, $M_i:=\bigoplus_{j\in\mathbb Z}M_{(i,j)}$ is finitely generated as a graded right $S_0:=\bigoplus_{j\in\mathbb Z}S_{(0,j)}$ module.

A homogeneous element $x$ of a bigraded module $M \in \BiGrMod S$ is called a \emph{locally torsion} element, if the graded right $S_0$-module $xS_0$ is finite dimensional.
A bigraded module $M \in \BiGrMod S$ is called a \emph{locally torsion} bigraded $S$-module, if all homogeneous elements of $M$ are locally torsion elements.
Write $\BiLTors S$ for the full subcategory of $\BiGrMod S$ consisting of all locally torsion bigraded $S$-modules, and $\lbiltors S$ for the full subcategory of $\lbigrmod S$ consisting of locally torsion objects.
Now $\BiLTors S$ is a Serre subcategory of $\BiGrMod S$ and $\lbiltors S$ is a Serre subcategory of $\lbigrmod S$. Hence we have the quotient categories
\[ \QBiGr_{\L} S:=\BiGrMod S/\BiLTors S \quad \text{and} \quad
\qlbigr_{\l} S :=\lbigrmod S/\lbiltors S.\]

We remark that all the results in the previous subsection have corresponding bigraded versions. Especially, we have the following version of Dade's Theorem for densely bigraded algebras. Since the proof of this result is similar to that of Theorem \ref{thm-d1}, we omit the proof.

\begin{thm}\label{thm-d2}
Let $S$ be as Convention \ref{cvn-biden}.
Then the natural functor $(-)_0: \BiGrMod S \longrightarrow \GrMod S_0; M\mapsto M_0$ induces the following equivalences of abelian categories:
$$(-)_0: \QBiGr_{\L} S\overset{\cong}{\longrightarrow} \QGr S_0
\quad \text{and} \quad
(-)_0: \qlbigr_{\l} S \overset{\cong}{\longrightarrow} \qgr S_0.$$
\end{thm}

\subsection{Smash products}

In this subsection, $A=\bigoplus_{i\in\mathbb N}A_i$ and $B=\bigoplus_{i\in\mathbb N}B_i$ are connected graded algebras such that $\dim A_i<\infty$ and $\dim B_i<\infty$ for all $i\in\mathbb N$. Assume that $A$ and $B$ are generated in degree 1.

Let $\psi:B\otimes A\to A\otimes B$ be a twisting map. Consider the smash product $A\#_\psi B$ (see Remark \ref{rem-tsp}(3)).
Let $S=A\#_\psi B$. We may endow the following bigrading on $S$:
\[S_{(i,j)}=A_{i+j}\otimes B_j \quad \text{for}\  i,j\in\mathbb Z.\]
In particular $S_0:=\bigoplus_{j\in\mathbb N}S_{(0,j)}$ is isomorphic to the twisted Segre product $A\circ_\psi B$.

\begin{prop}\label{prop-d1}
Let $S=A\#_\psi B$ be as above.
\begin{enumerate}
  \item $S$ is a densely bigraded algebra.
  \item Let $S_i:=\bigoplus_{j\in\mathbb N} S_{(i,j)}$ for all $i\in\mathbb Z$. Then $S_i$ is finitely generated both as a graded left $S_0$-module and as a graded right $S_0$-module.
\end{enumerate}
\end{prop}

\begin{proof}
(i) It is easy to see that $S$ is a bigraded algebra. For $i,s\in\mathbb Z$, $S_i=\bigoplus_{j\in\mathbb Z}A_{i+j}\otimes B_j$ and $S_s=\bigoplus_{t\in\mathbb Z}A_{s+t}\otimes B_t$. If $i,s\ge0$, then $A_i\otimes B_0\subseteq S_i$. Hence $$S_iS_s\supseteq \bigoplus_{t\in\mathbb Z}(A_i\otimes B_0)(A_{s+t}\otimes B_t)=\bigoplus_{t\in\mathbb Z}(A_{i+s+t}\otimes B_t)=S_{i+s},$$ where the first  equality follows from the condition that $A$ is generated in degree 1 and the twisting map $\psi$ satisfies the normal conditions (\ref{tm2}) and (\ref{tm3}). Hence $S_iS_s=S_{i+s}$ if $i,s\ge0$. Similarly, we see $S_iS_s=S_{i+s}$ if $i,s\leq 0$.

Now assume $i>0$ and $s<0$. Then $S_i=\bigoplus_{j\ge0}A_{i+j}\otimes B_j$ and $S_s=\bigoplus_{t\ge -s}A_{s+t}\otimes B_t$. Hence $$S_iS_s\supseteq \bigoplus_{t\ge-s}(A_i\otimes B_0)(A_{s+t}\otimes B_t)=\bigoplus_{t\ge-s}A_{i+s+t}\otimes B_t.$$ Therefore $S_{i+s}/S_iS_s$ must be finite dimensional since $\dim A_j<\infty$ and $\dim B_j<\infty$ for all $j$.

Similarly, $S_{i+s}/S_iS_s$ is finite dimensional if $i<0$ and $s>0$.

(ii) Assume $i\ge0$. Note that the twisting map $\psi$ is bijective, so
$\psi (B_j\otimes A_i)=A_i\otimes B_j$ for all $i,j$. We have, as a graded left $S_0$-module, \[S_i=\bigoplus_{j\ge0}A_{i+j}\otimes B_j=(\bigoplus_{j\ge0}A_{j}\otimes B_j)(A_i\otimes B_0)=S_0(A_i\otimes B_0).\]
Therefore, $S_i$ is finitely generated as a graded left $S_0$-module. Similarly, we may prove $S_i$ is finitely generated as a graded right $S_0$-module.

When $i<0$,
\[S_i=\bigoplus_{j\ge-i}A_{i+j}\otimes B_j=(\bigoplus_{j\ge0}A_{j}\otimes B_j)(A_0\otimes B_{-i})=S_0(A_0\otimes B_{-i}).\]
Hence $S_i$ is finitely generated as a graded left $S_0$-module. Similarly, we see that $S_i$ is finitely generated as a graded right $S_0$-module.
\end{proof}

\begin{rem}
Note that in \cite{VR}, a $\mathbb Z$-grading and an $\mathbb N \times \mathbb N$-bigrading were introduced for the tensor product $A\otimes B$.
We remark that the bigrading in this paper is different from that of \cite{VR}, and hence the corresponding categories of bigraded modules are also different.
\end{rem}

Combining Theorem \ref{thm-d2} and Proposition \ref{prop-d1}, we obtain the following result.

\begin{thm}\label{thm-d3} Assume that $A\circ_\psi B$ is a noetherian algebra. Then we have the following equivalence of abelian categories
\[\QBiGr_{\L} A\#_\psi B\cong \QGr A\circ_\psi B
\quad \text{and} \quad
\qlbigr_{\l} A\#_\psi B\cong \qgr A\circ_\psi B.\]
\end{thm}

\begin{rem}
In general it is not known whether the twisted Segre product $A\circ_\psi B$ is noetherian
when $A$ and $B$ are noetherian.
Note that if $S=A\#_\psi B$ is a noetherian algebra, then we see that $S_0=A\circ_\psi B$ is also a noetherian algebra, so we can use Theorem \ref{thm-d3}.
\end{rem}

\begin{rem}
By \eqref{tm2}, the map $A \to A \#_\psi B; a \mapsto a \otimes 1$ is an algebra homomorphism. This homomorphism induces an exact functor \[\GrMod A \longrightarrow \BiGrMod A \#_\psi B;\; M \mapsto M\otimes_A(A \#_\psi B),\] where
$(M\otimes_A(A \#_\psi B))_{(i,j)}=\sum_{p+q=i+j}M_p \otimes_A A_q \otimes B_{j}$. If $M \in \Tors A$, then $M\otimes_A(A \#_\psi B) \in \BiLTors A \#_\psi B$.
(Indeed, for any $m \in M$, there exist $s \in \NN$ such that $ma=0$ for all $a \in \bigoplus_{i \geq s}A_i$. Then for any $m \otimes a \otimes b \in M\otimes_A(A \#_\psi B)$ and any $c \otimes d \in \bigoplus_{i \geq s} A_i \otimes B_i$, we have
$(m \otimes a \otimes b)(c\otimes d)= m \otimes ac_\psi\otimes b^\psi d=
m \otimes ((ac_\psi \otimes 1)(1\otimes b^\psi d))=mac_\psi\otimes 1\otimes b^\psi d=0$, because $\deg ac_\psi \geq s$. Thus $m \otimes a \otimes b\in M\otimes_A(A \#_\psi B)$ is a locally torsion element.) Hence we obtain an induced functor $\QGr A \longrightarrow \QBiGr_{\L} A \#_\psi B$. If $A\circ_\psi B$ is a noetherian algebra, then by Theorem \ref{thm-d3}, we have a functor $\QGr A \longrightarrow \QGr A\circ_\psi B$, which sends $\pi A$ to $\pi (A\circ_\psi B)$. This means that there exists a map $\Proj A\circ_\psi B \to \Proj A$ in the sense of \cite[Section 2]{AZ} (see also \cite{Sm}). Similarly, we also have a map $\Proj A\circ_\psi B \to \Proj B$.
This consequence is analogous to \cite[Corollary 2.4]{VR}.
 \end{rem}

\subsection{Twisted Segre products of Koszul AS-regular algebras}

Now let $A$ and $B$ be noetherian Koszul AS-regular algebras, and let $\psi:B\otimes A\to A\otimes B$ be a twisting map. %The structure of the Yoneda algebra of
Then the smash product $A\#_\psi B$ has been studied in \cite{WS, SZL}.
Using the results for $A\#_\psi B$, we can establish the following theorem for the twisted Segre product $A\circ_\psi B$.

\begin{thm}\label{thm-d4}
Let $A$ and $B$ be noetherian Koszul AS-regular algebras, and let $\psi:B\otimes A\to A\otimes B$ be a twisting map.
Assume that  the twisted Segre product $A\circ_\psi B$ is noetherian.
Then $A\circ_\psi B$ is a noncommutative graded isolated singularity.
\end{thm}

\begin{proof}
Since $\psi$ is bijective, $\sigma_{\psi}$ is invertible, so the conditions in \cite[Theorem 2]{WS} (see also \cite[Theorem 2.11]{SZL}) are satisfied.
Hence $A\#_\psi B$, equipped with the $\mathbb N$-graded structure, is a Koszul algebra of finite global dimension $d=\gldim A+\gldim B$.
Thus we see that $A\#_\psi B$ has ungraded global dimension $d$.
Since the bigraded global dimension of $A\#_\psi B$ is not larger than $d$ by \cite[I.2.7]{NVO}, and the projection functor $\BiGrMod A\#_\psi B \longrightarrow \QBiGr_{\L} A\#_\psi B$ is exact and preserves injective objects by \cite[Corollary 5.4]{P},
it follows that the quotient categories $\QBiGr_{\L} A\#_\psi B$ and $\qlbigr_{\l} A\#_\psi B$ have finite global dimension.
In particular, we see that $\qgr A\circ_\psi B$ has finite global dimension by Theorem \ref{thm-d3}.
\end{proof}

% % % % % % % % % % % % % % % % % % % % %
\section{Twisted Segre products of quadratic algebras}
% % % % % % % % % % % % % % % % % % % % %
In this section, we study twisted Segre products of quadratic algebras.
Let $V$ and $U$ be finite dimensional vector spaces, and let $A=T(V)/(R_A)$ and $B=T(U)/(R_B)$ be quadratic algebras. Let $\psi:B\otimes A\to A\otimes B$ be a twisting map. We identify $A_1$ with $V$ and $B_1$ with $U$. Denote by $\psi_0: U\otimes V\to V\otimes U$  the restriction of $\psi$ to $B_1\otimes A_1$. By Remark \ref{rem-tsp}(2), $\psi_0$ determines uniquely a twisting map $\psi_T:T(U)\otimes T(V)\to T(V)\otimes T(U)$. By using Conditions (\ref{tm4}) and (\ref{tm5}), we can check that $\psi_T(R_B\otimes V)\subseteq V\otimes R_B$ and $\psi_T(U\otimes R_A)\subseteq R_A\otimes U$, which make the following diagram commutative
\begin{equation}\label{diag-qd}
  \xymatrix{
  T(U)\otimes T(V) \ar[d]_{\pi_B\otimes \pi_A} \ar[r]^{\psi_T} & T(V)\otimes T(U) \ar[d]^{\pi_A\otimes \pi_B} \\
  B\otimes A \ar[r]^{\psi} & A\otimes B,   }
\end{equation}
where the vertical maps are the natural projections. By Lemma \ref{lem-tspmorp}, we obtain a homomorphism of graded algebras
\begin{equation}\label{eq-mp1}
  h_{\pi_A\pi_B}:T(V)\circ_{\psi_T} T(U)\longrightarrow A\circ_\psi B.
\end{equation}
Combining the above homomorphism with the isomorphism $\Xi$ obtained in Remark \ref{rem-tsp}(2), we have
\begin{equation}\label{eq-mp2}
  \Theta:T(V\otimes U)\overset{\Xi}\longrightarrow T(V)\circ_{\psi_T} T(U)\overset{h_{\pi_A\pi_B}}\longrightarrow A\circ_\psi B.
\end{equation}
Since $h_{\pi_A\pi_B}$ is surjective, the composition homomorphism $\Theta$ is epic. We next compute the kernel of $\Theta$. Note that $$\Ker h_{\psi_A\pi_B}=(R_A)\circ_{\psi_T}  T(U)+T(V)\circ_{\psi_T} (R_B).$$
The ideal $\Ker h_{\psi_A\psi_B}$ is generated by $R_A\otimes (U\otimes U)+(V\otimes V)\otimes R_B$. Since $\Xi$ is an isomorphism, $\Ker \Theta$ is generated by
\begin{eqnarray*}
% \nonumber to remove numbering (before each equation)
  &\Xi^{-1}(R_A\otimes (U\otimes U)+(V\otimes V)\otimes R_B)\quad\qquad\qquad\\
  &=(1\otimes \psi_0^{-1}\otimes 1)(R_A\otimes (U\otimes U)+(V\otimes V)\otimes R_B).
\end{eqnarray*}
Summarizing the above narratives, we obtain the following result.

\begin{prop} Let $V$ and $U$ be finite dimensional vector spaces, and let $A=T(V)/(R_A)$ and $B=T(U)/(R_B)$ be quadratic algebras. Assume that  $\psi:B\otimes A \to A\otimes B$ is a twisting map. Let $\psi_0: U\otimes V\to V\otimes U$ be the restriction of $\psi$ to $U\otimes V$. Then $A\circ_\psi B\cong T(V\otimes U)/I$, where $I$ is the two-sided ideal of $T(V\otimes U)$ generated by $$(1\otimes \psi_0^{-1}\otimes 1)(R_A\otimes (U\otimes U)+(V\otimes V)\otimes R_B) \quad \subseteq (V\otimes U)\otimes (V\otimes U).$$
\end{prop}

We next give a condition to define a twisting map for quadratic algebras.

\begin{prop}\label{prop-quad} Let $A=T(V)/(R_A)$ and $B=T(U)/(R_B)$ be quadratic algebras, and let $\psi_0: U\otimes V\to V\otimes U$ be a bijective linear map.
Let $\psi_T:T(U)\otimes T(V)\to T(V)\otimes T(U)$ be the twisting map induced by $\psi_0$ (cf. Remark \ref{rem-tsp}(2)).
If $\psi_T$ satisfies
\[ \psi_T(R_B\otimes V)\subseteq V\otimes R_B  \quad  \text{and} \quad \psi_T(U\otimes R_A)\subseteq R_A\otimes U,\]
then it induces a twisting map $\psi:B\otimes A\to A\otimes B$ which fits into the commutative diagram (\ref{diag-qd}).
\end{prop}

\begin{proof} Let $I_A=(R_A)$ and $I_B=(R_B)$. Conditions (\ref{tm4}) and (\ref{tm5}) imply that $\psi_T(I_B\otimes T(V))\subseteq T(V)\otimes I_B$ and $\psi_T(T(U)\otimes I_A)\subseteq I_A\otimes T(U)$, which in turn imply that $\psi_T$ induces a twisting map $\psi:B\otimes A\to A\otimes B$.
\end{proof}

% % % % % % % % % % % % % % % % % % % % % %
\section{Twisted Segre products of $k[u,v]$ and $k[x,y]$} \label{sec-2poly}
% % % % % % % % % % % % % % % % % % % % % %
Let $A=k[u,v]$ and $B=k[x,y]$ be polynomial algebras of two variables.
In this section, we study twisted Segre products of $A$ and $B$.

\subsection{Twisting maps}
In this subsection, we calculate twisting maps for $A=k[u,v]$ and $B=k[x,y]$. Let $V=k u\oplus k v$ and $U=k x\oplus k y$,
and let $\psi_0:U\otimes V\to V\otimes U$ be a bijective linear map.
The map $\psi$ may be represented in the following way:
 \begin{equation}\label{twmm1}
   \psi_0(x\otimes \left(
                    \begin{array}{c}
                      u \\
                      v \\
                    \end{array}
                  \right)
   )=C\left(
                    \begin{array}{c}
                      u \\
                      v \\
                    \end{array}
                  \right)\otimes x+D\left(
                    \begin{array}{c}
                      u \\
                      v \\
                    \end{array}
                  \right)\otimes y,
 \end{equation}
 \begin{equation}\label{twmm2}
   \psi_0(y\otimes \left(
                    \begin{array}{c}
                      u \\
                      v \\
                    \end{array}
                  \right)
   )=P\left(
                    \begin{array}{c}
                      u \\
                      v \\
                    \end{array}
                  \right)\otimes x+Q\left(
                    \begin{array}{c}
                      u \\
                      v \\
                    \end{array}
                  \right)\otimes y,
 \end{equation}
 where $C,D,P,Q$ are $2\times 2$-matrices.

 Assume $C=(c_{ij})$, $D=(d_{ij})$, $P=(p_{ij})$ and $Q=(q_{ij})$. Consider the $4\times 4$-matrix
 $$H=\left(
      \begin{array}{cc}
        C & D \\
        P & Q \\
      \end{array}
    \right).
 $$
 Define a new matrix $S(H)$ by permuting entries of $H$ in the following way
 $$S(H)=\left(
          \begin{array}{cccc}
            c_{11} & d_{11} & c_{12} & d_{12} \\
            p_{11} & q_{11} & p_{12} & q_{12} \\
            c_{21} & d_{21} & c_{22} & d_{22} \\
            p_{21} & q_{21} & p_{22} & q_{22} \\
          \end{array}
        \right).
 $$
 Set $$\widetilde{C}=\left(
                      \begin{array}{cc}
                        c_{11} & d_{11}  \\
                        p_{11} & q_{11}
                      \end{array}
                    \right),
 \widetilde{D}=\left(
                      \begin{array}{cc}
                        c_{12} & d_{12} \\
                       p_{12} & q_{12}
                      \end{array}
                    \right),$$
                     $$\widetilde{P}= \left(
                                       \begin{array}{cc}
                                          c_{21} & d_{21}  \\
                                         p_{21} & q_{21}
                                       \end{array}
                                     \right), \widetilde{Q}=\left(
                                            \begin{array}{cc}
                                              c_{22} & d_{22} \\
                                             p_{22} & q_{22}
                                            \end{array}
                                          \right).$$
 Then we have
 \begin{equation}\label{twmm3}
   \psi_0( \left(
                    \begin{array}{c}
                      x \\
                      y \\
                    \end{array}
                  \right)\otimes u
   )=u\otimes\widetilde{C}\left(
                    \begin{array}{c}
                      x \\
                      y \\
                    \end{array}
                  \right) +v\otimes \widetilde{D}\left(
                    \begin{array}{c}
                      x \\
                      y \\
                    \end{array}
                  \right),
 \end{equation}
 \begin{equation}\label{twmm4}
   \psi_0(\left(
                    \begin{array}{c}
                      x \\
                      y \\
                    \end{array}
                  \right)\otimes v
   )=u\otimes \widetilde{P}\left(
                    \begin{array}{c}
                      x \\
                      y \\
                    \end{array}
                  \right)+v\otimes \widetilde{Q}\left(
                    \begin{array}{c}
                      x \\
                      y \\
                    \end{array}
                  \right).
 \end{equation}

 \begin{rem} (1) We indeed obtain a map $S:M_4(k)\to M_4(k)$, where $M_4(k)$ is the algebra of $4\times 4$-matrices over $k$. By the above definition, $S^2(H)=H$ for every $H\in M_4(k)$.

 (2) $S$ preserves the multiplication of matrices, and hence $S$ is an algebra automorphism of $M_4(k)$.
 \end{rem}

 Let $\psi:B\otimes A \to A\otimes B$ be a twisting map such that $\psi|_{U\otimes V}=\psi_0$. The we have the following equalities:
 \begin{equation}\label{eq-ctw1}
   \psi(xy\otimes \left(
                    \begin{array}{c}
                      u \\
                      v \\
                    \end{array}
                  \right)
   )=\psi(yx\otimes \left(
                    \begin{array}{c}
                      u \\
                      v \\
                    \end{array}
                  \right)),
 \end{equation}

 \begin{equation}\label{eq-ctw2}
   \psi(\left(
                    \begin{array}{c}
                      x \\
                      y \\
                    \end{array}
                  \right)\otimes uv)=\psi(\left(
                    \begin{array}{c}
                      x \\
                      y \\
                    \end{array}
                  \right)\otimes vu).
 \end{equation}

 By Equations (\ref{twmm1}) and (\ref{twmm2}), we have
 \begin{eqnarray*}
 % \nonumber to remove numbering (before each equation)
   \psi(xy\otimes \left(
                    \begin{array}{c}
                      u \\
                      v \\
                    \end{array}
                  \right)
   )=&PC\left(
                    \begin{array}{c}
                      u \\
                      v \\
                    \end{array}
                  \right)\otimes x^2+PD\left(
                    \begin{array}{c}
                      u \\
                      v \\
                    \end{array}
                  \right)\otimes yx\\
                  &+QC\left(
                    \begin{array}{c}
                      u \\
                      v \\
                    \end{array}
                  \right)\otimes xy+ QD\left(
                    \begin{array}{c}
                      u \\
                      v \\
                    \end{array}
                  \right)\otimes y^2,
 \end{eqnarray*}
 \begin{eqnarray*}
 % \nonumber to remove numbering (before each equation)
   \psi(yx \otimes \left(
                    \begin{array}{c}
                      u \\
                      v \\
                    \end{array}
                  \right)
   )=&CP\left(
                    \begin{array}{c}
                      u \\
                      v \\
                    \end{array}
                  \right)\otimes x^2+DP\left(
                    \begin{array}{c}
                      u \\
                      v \\
                    \end{array}
                  \right)\otimes xy\\
                  &+CQ\left(
                    \begin{array}{c}
                      u \\
                      v \\
                    \end{array}
                  \right)\otimes yx+D Q\left(
                    \begin{array}{c}
                      u \\
                      v \\
                    \end{array}
                  \right)\otimes y^2.
 \end{eqnarray*}
 Since $xy=yx$, Equation (\ref{eq-ctw1}) implies $$CP=PC, DQ=QD, DP+CQ=PD+QC.$$
 Similarly, Equations (\ref{twmm1}), (\ref{twmm2}) and (\ref{eq-ctw2}) imply $$\widetilde{C}\widetilde{P}=\widetilde{P}\widetilde{C}, \widetilde{D}\widetilde{Q}=\widetilde{Q}\widetilde{D}, \widetilde{D}\widetilde{P}+\widetilde{C}\widetilde{Q}=\widetilde{P}\widetilde{D}+\widetilde{Q}\widetilde{C}.$$

 Summarizing the above narratives, we obtain the following result.

 \begin{prop}\label{prop-twmc}
 Let $A=k[u,v], B=k[x,y]$ and let $V=k u\oplus k v, U=k x\oplus k y$. Assume that $\psi_0:U\otimes V\to V\otimes U$ is a bijective linear map, and $\psi_0$ is represented by (\ref{twmm1}) and (\ref{twmm2}). Then $\psi_0$ uniquely determines a twisting map $\psi:B\otimes A\to A\otimes B$ if and only if
 \begin{equation}\label{eq-cond1}
   CP=PC, DQ=QD, DP+CQ=PD+QC,
 \end{equation}
 and
 \begin{equation}\label{eq-cond2}
   \widetilde{C}\widetilde{P}=\widetilde{P}\widetilde{C}, \widetilde{D}\widetilde{Q}=\widetilde{Q}\widetilde{D}, \widetilde{D}\widetilde{P}+\widetilde{C}\widetilde{Q}=\widetilde{P}\widetilde{D}+\widetilde{Q}\widetilde{C}.
 \end{equation}
 \end{prop}

\subsection{Diagonal twisting maps}
In this subsection, we focus on a special case of twisting maps for $A=k[u,v]$ and $B=k[x,y]$,
specifically the case when twisting maps are diagonal (cf. Definition \ref{def-diag}).
By Proposition \ref{prop-twmc}, we have the following.

\begin{prop}\label{prop-dtwmc}
Let $A=k[u,v], B=k[x,y]$ and let $V=k u\oplus k v, U=k x\oplus k y$. Assume that $\psi_0:U\otimes V\to V\otimes U$ is a bijective linear map, and $\psi_0$ is represented by (\ref{twmm1}) and (\ref{twmm2}). Then $\psi_0$ uniquely determines a diagonal twisting map $\psi:B\otimes A\to A\otimes B$ if and only if
\begin{equation}\label{eq-condd1}
   D=P=0 \quad \text{and} \quad  CQ=QC.
\end{equation}
% Equations in (\ref{eq-cond2}) are automatically satisfied).
\end{prop}

Thus if $\psi:B\otimes A\to A\otimes B$ is a diagonal twisting map, then $\psi_0$ is given by
 \begin{equation}
   \psi_0(x\otimes \left(
                    \begin{array}{c}
                      u \\
                      v \\
                    \end{array}
                  \right)
   )=C\left(
                    \begin{array}{c}
                      u \\
                      v \\
                    \end{array}
                  \right)\otimes x,
 \end{equation}
 \begin{equation}
   \psi_0(y\otimes \left(
                    \begin{array}{c}
                      u \\
                      v \\
                    \end{array}
                  \right)
   )=Q\left(
                    \begin{array}{c}
                      u \\
                      v \\
                    \end{array}
                  \right)\otimes y,
 \end{equation}
where $C, Q$ are invertible $2\times 2$-matrices with $CQ=QC$.

In this case, since $CQ=QC$, there exists an invertible matrix $X$ such that
\[X^{-1}CX=\left(\begin{array}{cc}
a_{11} & 0 \\
a_{21} & a_{22} \\
 \end{array}
\right)
\quad \text{and} \quad
X^{-1}QX=\left(
 \begin{array}{cc}
b_{11} & 0 \\
b_{21} & b_{22} \\
 \end{array}
 \right).
\]
%Since $C$ and $Q$ are invertible, $a_{11}a_{22}\neq0, b_{11}b_{22}\neq0$.
Define an automorphism $\sigma\in \Aut(k[u,v])$ by setting
\[\sigma\left(
          \begin{array}{c}
            u \\
            v \\
          \end{array}
        \right)=X\left(
          \begin{array}{c}
            u \\
            v \\
          \end{array}
        \right),
\]
 and a linear bijective map
$\psi'_0:U\otimes V\to V\otimes U$ by setting
 \begin{equation}\label{eq-twsm1}
 \psi'_0(x\otimes\left(
          \begin{array}{c}
            u \\
            v \\
          \end{array}
        \right))=
        X^{-1}CX
        \left(
          \begin{array}{c}
            u \\
            v \\
          \end{array}
        \right)\otimes x,
 \end{equation}
 \begin{equation}\label{eq-twsm2}
       \psi'_0(y\otimes\left(
          \begin{array}{c}
            u \\
            v \\
          \end{array}
        \right))=
        X^{-1}QX
        \left(
          \begin{array}{c}
            u \\
            v \\
          \end{array}
        \right)\otimes y.
\end{equation}
Since $(X^{-1}CX)(X^{-1}QX)=(X^{-1}QX)(X^{-1}CX)$, it follows from Proposition \ref{prop-twmc} that $\psi'_0$ determines a twisting map
\[\psi':B\otimes A\longrightarrow A\otimes B.\]

We have the following computations
 \begin{eqnarray*}
 % \nonumber to remove numbering (before each equation)
   (\sigma\otimes\id)\psi_0(x\otimes\left(
          \begin{array}{c}
            u \\
            v \\
          \end{array}
        \right)) = (\sigma\otimes\id)(C\left(
          \begin{array}{c}
            u \\
            v \\
          \end{array}
        \right)\otimes x)
   = CX\left(
          \begin{array}{c}
            u \\
            v \\
          \end{array}
        \right)\otimes x,
 \end{eqnarray*}
 \begin{eqnarray*}
 % \nonumber to remove numbering (before each equation)
   \psi'_0(\id\otimes\sigma)(x\otimes\left(
          \begin{array}{c}
            u \\
            v \\
          \end{array}
        \right)) = \psi'_0(x\otimes X\left(
          \begin{array}{c}
            u \\
            v \\
          \end{array}
        \right))
    = X(X^{-1}CX)
    \left(
          \begin{array}{c}
            u \\
            v \\
          \end{array}
        \right)\otimes x.
 \end{eqnarray*}
Thus we have
\[(\sigma\otimes \id)\psi_0(x\otimes\left(
          \begin{array}{c}
            u \\
            v \\
          \end{array}
        \right))=\psi'_0(\id\otimes\sigma)(x\otimes\left(
          \begin{array}{c}
            u \\
            v \\
          \end{array}
        \right)).
\]
Similarly, we have
\[(\sigma\otimes \id)\psi_0(y\otimes\left(
          \begin{array}{c}
            u \\
            v \\
          \end{array}
        \right))=\psi'_0(\id\otimes\sigma)(y\otimes\left(
          \begin{array}{c}
            u \\
            v \\
          \end{array}
        \right)).
\]
Therefore, we have $(\sigma\otimes \id)\psi_0=\psi'_0(\id\otimes\sigma)$.
Since the twisting maps $\psi$ and $\psi'$ are determined by $\psi_0$ and $\psi'_0$ respectively, it follows that $(\sigma\otimes \id)\psi=\psi'(\id\otimes\sigma)$.
By Lemma \ref{lem-tspmorp}, we obtain an isomorphism
\[h_{\sigma,\id}: A\circ_\psi B\overset{\cong}\longrightarrow A\circ_{\psi'}B.\]

Considering this fact, up to isomorphism of twisted Segre products, we may assume that
\[
C=
\left(\begin{array}{cc}
a_{11} & 0 \\
a_{21} & a_{22} \\
 \end{array}
\right)
\quad \text{and} \quad
Q=
\left(\begin{array}{cc}
b_{11} & 0 \\
b_{21} & b_{22} \\
 \end{array}
\right)
\]
are invertible lower triangular matrices such that $CQ=QC$.

\subsection{Defining relations}
Let $A=k[u,v], B=k[x,y]$ and let $\psi:B\otimes  A\to A\otimes B$ be a diagonal twisting map.
By the discussion in the previous subsection, we may assume that $\psi_0$ is given by
 \begin{equation}
   \psi_0(x\otimes \left(
                    \begin{array}{c}
                      u \\
                      v \\
                    \end{array}
                  \right)
   )=C
   \left(
                    \begin{array}{c}
                      u \\
                      v \\
                    \end{array}
                  \right)\otimes x,
\end{equation}
\begin{equation}
   \psi_0(y\otimes \left(
                    \begin{array}{c}
                      u \\
                      v \\
                    \end{array}
                  \right)
   )=Q\left(
                    \begin{array}{c}
                      u \\
                      v \\
                    \end{array}
                  \right)\otimes y,
\end{equation}
where
\[
C=
\left(\begin{array}{cc}
a_{11} & 0 \\
a_{21} & a_{22} \\
 \end{array}
\right)
\quad \text{and} \quad
Q=
\left(\begin{array}{cc}
b_{11} & 0 \\
b_{21} & b_{22} \\
 \end{array}
\right)
\]
are invertible lower triangular matrices such that $CQ=QC$. Thus we have $a_{11}a_{22} \neq0$, $b_{11}b_{22}\neq0$ and
\begin{align}\label{eq.comm}
a_{11}b_{21}+a_{21}b_{22}=a_{21}b_{11}+a_{22}b_{21}.
\end{align}

The aim of this subsection is to show the following.

\begin{thm} \label{thm.rel}
Let $A$, $B$, and $\psi$ be as above.

\begin{enumerate}
\item $A\circ_{\psi} B$ is presented by
\[k\<X,Y,Z,W\>/(f_1, f_2, f_3, f_4, f_5, f_6, f_7),\]
where
\begin{align*}
&f_1=a_{11}XY-a_{22}YX-a_{21}X^2,\\
&f_2= a_{11} ZX -b_{11}XZ,\\
&f_3= a_{11} ZY -b_{21}XZ-b_{22}YZ,\\
&f_4= a_{11}a_{22}  WX +a_{21}b_{11}XZ-a_{11}b_{11}XW,\\
&f_5= a_{11}a_{22}  WY +a_{21}b_{21}XZ+a_{21}b_{22}YZ-a_{11}b_{21}XW-a_{11}b_{22}YW,\\
&f_6=b_{11}ZW-b_{22}WZ-b_{21}Z^2,\\
&f_7=a_{11}XW-a_{22}YZ-a_{21}XZ.
\end{align*}
\item $A\circ_{\psi} B$ is a noncommutative quadric surface.
In particular, it is noetherian, Koszul, and AS-Gorenstein.
\end{enumerate}
\end{thm}

To prove this theorem, we prepare two lemmas.

\begin{lem} \label{lem.ASrel}
If $S$ is $k\<X,Y,Z,W\>/(f_1, f_2, f_3, f_4, f_5, f_6),$ where
\begin{align*}
&f_1=a_{11}XY-a_{22}YX-a_{21}X^2,\\
&f_2= a_{11} ZX -b_{11}XZ,\\
&f_3= a_{11} ZY -b_{21}XZ-b_{22}YZ,\\
&f_4= a_{11}a_{22}  WX +a_{21}b_{11}XZ-a_{11}b_{11}XW,\\
&f_5= a_{11}a_{22}  WY +a_{21}b_{21}XZ+a_{21}b_{22}YZ-a_{11}b_{21}XW-a_{11}b_{22}YW,\\
&f_6=b_{11}ZW-b_{22}WZ-b_{21}Z^2,
\end{align*}
then $S$ is a $4$-dimensional noetherian Koszul AS-regular domain.
\end{lem}

\begin{proof}
It is easy to see that $S$ is isomorphic to $k\<X,Y,Z,W\>/(f_1', f_2',f_3',f_4', f_5', f_6'),$ where
\begin{align*}
&f_1'=a_{11}'XY-a_{22}'YX-a_{21}'X^2,\\
&f_2'= a_{11}' ZX -b_{11}'XZ,\\
&f_3'= a_{11}' ZY -b_{21}'XZ-YZ,\\
&f_4'= a_{11}'a_{22}'  WX +a_{21}'b_{11}'XZ-a_{11}'b_{11}'XW,\\
&f_5'= a_{11}'a_{22}'WY +a_{21}'b_{21}'XZ+a_{21}'YZ-a_{11}'b_{21}'XW-a_{11}'YW,\\
&f_6'=b_{11}'ZW-WZ-b_{21}'Z^2,
\end{align*}
with $a_{ij}'=a_{ij}b_{22}^{-1}$ and $b_{ij}'=b_{ij}b_{22}^{-1}$.
Therefore, in the rest of this proof,
we assume that $b_{22}=1$ without loss of generality.

One can easily check that $T=k\<X,Y\>/(f_1)$ is a $2$-dimensional noetherian AS-regular algebra. We now verify that $S$ is a double extension of $T$ in the sense of \cite{ZZ1, ZZ2}.
Since $CQ=QC$, the assignment
\[\phi\left(
          \begin{array}{c}
            X \\
            Y \\
          \end{array}
        \right)=Q\left(
          \begin{array}{c}
            X \\
            Y \\
          \end{array}
        \right)
        =\left(\begin{array}{cc}
        b_{11} & 0 \\
        b_{21} & 1 \\
         \end{array}\right)\left(
                  \begin{array}{c}
                    X \\
                    Y \\
                  \end{array}
                \right)
\]
defines a graded algebra automorphism of $T$.
We define an algebra  homomorphism $\sigma: T \to M_2(T)$ by
\[ \sigma
=\begin{pmatrix} \sigma_{11} &0 \\ \sigma_{21} &\sigma_{22} \end{pmatrix}
=\begin{pmatrix} a_{11}^{-1}\phi &0 \\ -a_{21}(a_{11}a_{22})^{-1} \phi  &a_{22}^{-1} \phi  \end{pmatrix}.
\]
Then we can rewrite $f_2,f_3,f_4,f_5$ as
\begin{align*}
&ZX=  a_{11}^{-1}b_{11}XZ = \sigma_{11}(X)Z,\\
&ZY=  a_{11}^{-1}(b_{21}XZ+YZ)=\sigma_{11}(Y)Z,\\
&WX=(a_{11}a_{22})^{-1}(-a_{21}b_{11}XZ+a_{11}b_{11}XW)=\sigma_{21}(X)Z+\sigma_{22}(X)W,\\
&WY=(a_{11}a_{22})^{-1}(-a_{21}b_{21}XZ-a_{21}YZ+a_{11}b_{21}XW+a_{11}YW)
=\sigma_{21}(Y)Z+\sigma_{22}(Y)W.
\end{align*}
Using  (\ref{eq.comm}), under the assumption that $f_2=0,f_3=0,\cdots,f_6=0$, we can check that
the resulting elements obtained from resolving $(WZ)X$ and from resolving
$W(ZX)$ are same,
and the resulting elements obtained from resolving $(WZ)Y$ and from resolving
$W(ZY)$ are same in the sense of \cite[Section 1]{ZZ1}. Thus \cite[Lemma 1.10 (c) and Proposition 1.11]{ZZ1} implies that $S$ is the right double extension of $T$ associated with $\sigma$.
Since
\begin{align*}
&\begin{pmatrix} a_{11}^{-1}\phi&0 \\ -a_{21}(a_{11}a_{22})^{-1}\phi &a_{22}^{-1} \phi \end{pmatrix}
\begin{pmatrix} a_{11}\phi^{-1}  &0  \\ a_{21}\phi^{-1}  &a_{22}\phi^{-1}  \end{pmatrix}\\
&=
\begin{pmatrix} a_{11}\phi^{-1}  &0  \\ a_{21}\phi^{-1}  &a_{22}\phi^{-1}  \end{pmatrix}
\begin{pmatrix} a_{11}^{-1}\phi &0\\ -a_{21}(a_{11}a_{22})^{-1}\phi &a_{22}^{-1}\phi \end{pmatrix}
=\begin{pmatrix}\id_T  &0  \\ 0  &\id_T  \end{pmatrix},
\end{align*}
$\sigma$ is invertible in the sense of \cite{ZZ1, ZZ2}, so it follows from \cite[Proposition 1.13]{ZZ1} that $S$ is the double extension of $T$ associated with $\sigma$.
Hence $T$ is a strongly noetherian, Koszul, Auslander regular and Cohen-Macaulay domain of global dimension $4$ by \cite[Theorem 0.1]{ZZ2}. In particular, it is a noetherian AS-regular algebra.
\end{proof}

\begin{lem} \label{lem.nor}
Let $S$ be as in Lemma \ref{lem.ASrel}. Then $f_7=a_{11}XW-a_{22}YZ-a_{21}XZ$ is a regular normal element of $S$.
\end{lem}

\begin{proof}
Since $S$ is a domain by Lemma \ref{lem.ASrel}, $f_7$ is a regular element.
Moreover, one can check
\begin{align*}
f_7X &= \frac{b_{11}}{a_{22}}Xf_7,\\
f_7Y &= \frac{(a_{11}b_{21}+a_{21}b_{22})X+a_{22}b_{22}Y}{a_{11}a_{22}}f_7,\\
f_7Z &= \frac{a_{11}}{b_{22}}Zf_7,\\
f_7W &= \frac{(a_{11}b_{21}+a_{21}b_{22})Z+a_{22}b_{22}W}{b_{11}b_{22}}f_7,
\end{align*}
so $f_7$ is a normal element. Indeed, we have
\begin{align*}
&(a_{11}XW-a_{22}YZ-a_{21}XZ)X\\
&=-a_{22}^{-1}X(a_{21}b_{11}XZ-a_{11}b_{11}XW)
-a_{22}b_{11}a_{11}^{-1}YXZ -a_{21}b_{11}a_{11}^{-1}XXZ\\
&=-a_{22}^{-1}X(a_{21}b_{11}XZ-a_{11}b_{11}XW) -b_{11}a_{11}^{-1}(a_{11}XY-a_{21}X^2)Z -a_{21}b_{11}a_{11}^{-1}XXZ\\
&=\frac{b_{11}}{a_{22}}X(a_{11}XW-a_{22}YZ-a_{21}XZ)
\end{align*}
and
\begin{align*}
&(a_{11}XW-a_{22}YZ-a_{21}XZ)Y\\
&=a_{22}^{-1}X(X(-a_{21}b_{21}Z+a_{11}b_{21}W)+Y(-a_{21}b_{22}Z+a_{11}b_{22}W))\\
&\qquad\qquad -a_{22}a_{11}^{-1}Y(b_{21}XZ+b_{22}YZ) -a_{21}a_{11}^{-1}X(b_{21}XZ+b_{22}YZ)\\
&=a_{22}^{-1}X^2(-a_{21}b_{21}Z+a_{11}b_{21}W)
+(a_{11}a_{22})^{-1}b_{22}(a_{22}YX+a_{21}X^2)(-a_{21}Z+a_{11}W)\\
&\qquad\qquad +b_{21}a_{11}^{-1}(a_{11}XY-a_{21}X^2)Z-a_{22}b_{22}a_{11}^{-1}Y^2Z -a_{21}a_{11}^{-1}X(b_{21}XZ+b_{22}YZ)\\
&=(b_{21}a_{22}^{-1}+a_{21}b_{22}(a_{11}a_{22})^{-1})a_{11}X^2W
+b_{22}YXW\\
&\qquad\qquad -(b_{21}a_{22}^{-1}+a_{21}b_{22}(a_{11}a_{22})^{-1})a_{22}XYZ
-a_{11}^{-1}b_{22}a_{22}Y^2Z\\
&\qquad\qquad\qquad\qquad -(b_{21}a_{22}^{-1}+a_{21}b_{22}(a_{11}a_{22})^{-1})a_{21}X^2Z
-a_{11}^{-1}b_{22}a_{21}YXZ\\
&=\frac{(a_{11}b_{21}+a_{21}b_{22})X+a_{22}b_{22}Y}{a_{11}a_{22}}(a_{11}XW-a_{22}YZ-a_{21}XZ).
\end{align*}
In a similar way, using (\ref{eq.comm}), we can verify that the other equations also hold.
\end{proof}

\begin{proof}[\textit{\textbf{Proof of Theorem \ref{thm.rel}}}]
Let $\pi: k\<X,Y,Z,W\> \to A\circ_{\psi} B$ be the surjection defined by
\[ X \mapsto u\otimes x, \quad Y \mapsto v\otimes x,\quad  Z \mapsto u\otimes y, \quad W \mapsto v\otimes y.\]
Since
\begin{align*}
\pi(f_1) &= a_{11}(u\otimes x)(v\otimes x)-a_{22}(v\otimes x)(u\otimes x)-a_{21}(u\otimes x)(u\otimes x)\\
&=a_{11}(u(a_{21}u+a_{22}v)\otimes x^2)-a_{22}(v(a_{11}u)\otimes x^2)-a_{21}(u(a_{11}u)\otimes x^2)\\
&=(a_{11}u(a_{21}u+a_{22}v)
-a_{11}a_{22}uv
-a_{21}a_{11}u^2
)\otimes x^2=0,
\end{align*}
we have $f_1 \in \Ker \pi$.
Moreover, since
\begin{align*}
\pi(f_2)
&= -a_{11}(u\otimes y)(u\otimes x)
+b_{11}(u\otimes x)(u\otimes y)\\
&= -a_{11}(u(b_{11}u)\otimes yx)
+b_{11}(u(a_{11}u)\otimes xy)\\
&= (-a_{11}b_{11}u^2+b_{11}a_{11}u^2)\otimes xy=0,
\end{align*}
we have $f_2 \in \Ker \pi$. Similarly, one can check $f_3,\dots, f_7 \in \Ker \pi$, so
$(f_1, \dots, f_7) \subseteq \Ker \pi$.
By Lemmas \ref{lem.ASrel} and \ref{lem.nor}, we see that the Hilbert series of $k\<X,Y,Z,W\>/(f_1, \dots, f_7)$ is $(1+t)(1-t)^{-3}$,
so the Hilbert series of $k\<X,Y,Z,W\>/(f_1, \dots, f_7)$ and $ A\circ_{\psi} B$ coincide. Therefore, $\pi$ induces an isomorphism
 $k\<X,Y,Z,W\>/(f_1, \dots, f_7) \xrightarrow{\sim} A\circ_{\psi} B$.
Hence the assertion \textrm{(i)} holds. Furthermore, the assertion \textrm{(ii)} directly follows from Lemmas \ref{lem.ASrel} and \ref{lem.nor}.
\end{proof}

By observing the following simple example, we can see that the twisted Segre product is a proper (and probably very wide) extension of the usual Segre product (compare with Proposition \ref{prop.spas}).

\begin{ex}\label{ex.nqs}
Let us consider the case $C= \begin{pmatrix} a_{11} &0 \\ 0 &a_{22} \end{pmatrix}$ and  $Q= \begin{pmatrix} b_{11} &0 \\ 0 &b_{22} \end{pmatrix}$.
In this case, $A\circ_{\psi} B$ is presented by $k\<X,Y,Z,W\>/(f_1, f_2, f_3, f_4, f_5, f_6, f_7)$, where
\begin{align*}
&f_1=a_{11}XY-a_{22}YX,\\
&f_2=a_{11}ZX-b_{11}XZ,\\
&f_3=a_{22}WX-b_{11}XW,\\
&f_4=a_{11}ZY-b_{22}YZ,\\
&f_5=a_{22}WY-b_{22}YW,\\
&f_6=b_{11}ZW-b_{22}WZ,\\
&f_7=a_{11}XW-a_{22}YZ
\end{align*}
with $a_{11}a_{22} \neq 0, b_{11}b_{22}\neq 0$.
Then we have
\[ \GrMod A\circ_{\psi} B  \cong \GrMod A\circ B \quad \Longleftrightarrow \quad a_{11}b_{22}= a_{22}b_{11}.\]

($\Leftarrow$) We have that $\phi(X) = b_{11}^{-1}X, \phi(Y) = b_{22}^{-1}Y, \phi(Z) = a_{11}^{-1}Z, \phi(W) = a_{22}^{-1}W$ defines an automorphism $\phi$ of $A\circ_{\psi} B$. Moreover, the Zhang twist $(A\circ_{\psi} B)^{\phi}$ is isomorphic to $A\circ B$ (see \cite{Z} for Zhang twists). Indeed, one can check
\begin{align*}
&X*Y = X\phi(Y)= \frac{1}{b_{22}}XY=  \frac{a_{22}}{a_{11}b_{22}}YX= \frac{1}{b_{11}}YX = Y\phi(X)= Y*X,\\
&Z*X = Z\phi(X)= \frac{1}{b_{11}}ZX =\frac{b_{11}}{a_{11}b_{11}}XZ = \frac{1}{a_{11}}XZ= X\phi(Z)= X*Z,
\end{align*}
and so on. Hence it follows that $\GrMod A\circ_{\psi} B  \cong \GrMod A\circ B$ by \cite[Theorem 1.1]{Z}.

($\Rightarrow$) We now assume that $\GrMod A\circ_{\psi} B  \cong \GrMod A\circ  B$ and $a_{11}b_{22} \neq a_{22}b_{11}$ and show that it induces a contradiction.
Since $A\circ  B$ is commutative and $\GrMod A\circ_{\psi} B  \cong \GrMod A\circ  B$,
the point scheme of $A\circ_{\psi} B$ is isomorphic to $\Proj (A\circ  B) \cong V(XW-YZ) \subset  \PP^3$ (a smooth quadric surface).
Let $S= k\<X,Y,Z,W\>/(f_1, f_2, f_3, f_4, f_5, f_6)$. Since a point module over $A\circ_{\psi} B$ can be regarded as a point module over $S$, the point scheme of $A\circ_{\psi} B$ is contained in the point scheme of $S$. On the other hand, the condition $\frac{a_{11}b_{22}}{a_{22}b_{11}}\neq 1$ implies that the point scheme of $S$ is isomorphic to
$V (X,Y) \cup V (X,Z) \cup V (X,W) \cup V (Y,Z) \cup V (Y,W) \cup V (Z,W) \subset  \PP^3$
(a union of six lines) by \cite[Proposition 4.2]{Vi} or \cite[Theorem 1(1)]{BDL}. Thus a smooth quadric surface is included in a union of six lines, which is a contradiction.
\end{ex}

\subsection{Stable categories of graded maximal Cohen-Macaulay modules}
We continue to use the notation of the previous subsection, that is, $A=k[u,v], B=k[x,y]$, and $\psi:B\otimes  A\to A\otimes B$ is a diagonal twisting map induced by
 \begin{equation}
   \psi_0(x\otimes \left(
                    \begin{array}{c}
                      u \\
                      v \\
                    \end{array}
                  \right)
   )=\left(\begin{array}{cc}
   a_{11} & 0 \\
   a_{21} & a_{22} \\
    \end{array}
   \right)
   \left(
                    \begin{array}{c}
                      u \\
                      v \\
                    \end{array}
                  \right)\otimes x,
\end{equation}
\begin{equation}
   \psi_0(y\otimes \left(
                    \begin{array}{c}
                      u \\
                      v \\
                    \end{array}
                  \right)
   )=\left(\begin{array}{cc}
   b_{11} & 0 \\
   b_{21} & b_{22} \\
    \end{array}
   \right)\left(
                    \begin{array}{c}
                      u \\
                      v \\
                    \end{array}
                  \right)\otimes y,
\end{equation}
where $a_{11}a_{22} \neq0$, $b_{11}b_{22}\neq0$ and
$a_{21}b_{11}+a_{22}b_{21}=b_{21}a_{11}+b_{22}a_{21}$.
In this subsection, we compute $\uCM^{\mathbb Z}(A\circ_{\psi} B)$.

\begin{lem}\label{lem.Kos}
Let $A=k[u,v], B=k[x,y]$, and let $\psi$ be as in the beginning of this subsection.
Then the Koszul dual $(A\circ_{\psi} B)^!$ is presented by
\[k\<X,Y,Z,W\>/(g_1, g_2, g_3, g_4, g_5, g_6, g_7, g_8, g_9),\]
where
\begin{align*}
&g_1=a_{11}YX+a_{22}XY,\\
&g_2=a_{11}XZ+a_{21}XW+b_{11}ZX+b_{21}ZY,\\
&g_3=b_{22}WY+a_{22}YW,\\
&g_4=b_{22}ZW+b_{11}WZ,\\
&g_5=a_{11}X^2+a_{21}XY,\\
&g_6=Y^2,\\
&g_7=b_{11}Z^2+b_{21}ZW,\\
&g_8=W^2,\\
&g_9=b_{22}ZY+a_{21}YW+a_{11}YZ+a_{22}XW+b_{11}WX+b_{21}WY.
\end{align*}
\end{lem}

\begin{proof}
This follows by a straightforward calculation.
\end{proof}

\begin{lem}\label{lem.nord}
In the setting of Lemma \ref{lem.Kos}, $w=b_{22}ZY+a_{21}YW+a_{11}YZ$ is a regular normal element of $(A\circ_{\psi} B)^!=k\<X,Y,Z,W\>/(g_1, g_2, \dots, g_9)$ such that
\[ (A\circ_{\psi} B)^!/(w)\cong  k\<X,Y,Z,W\>/(g_1, g_2, \dots , g_9, w) \cong S^!,\]
where $S$ is as in Lemma \ref{lem.ASrel}.
\end{lem}

\begin{proof}
One can verify that
\begin{align*}
wX &= \frac{a_{22}b_{22}X-(a_{11}b_{21}+a_{21}b_{22})Y}{b_{11}b_{22}}w,\\
wY &=\frac{a_{11}}{b_{22}}Yw,\\
wZ &= \frac{a_{22}b_{22}Z-(a_{11}b_{21}+a_{21}b_{22})W}{a_{11}a_{22}}w,\\
wW &= \frac{b_{11}}{a_{22}}Ww,
\end{align*}
so $w$ is a normal element of $k\<X,Y,Z,W\>/(g_1,\dots, g_9)$.
Indeed, we have
\begin{align*}
&(b_{22}ZY+a_{21}YW+a_{11}YZ)X\\
&=-(a_{22}XW+b_{11}WX+b_{21}WY)X\\
&=-(a_{22}XWX-b_{11}a_{21}a_{11}^{-1}WXY-b_{21}a_{22}b_{22}^{-1}YWX)\\
&=-(a_{22}XWX-b_{11}a_{21}b_{22}^{-1}YWX-b_{21}a_{22}b_{22}^{-1}YWX)\\
&=-a_{22}XWX+(b_{11}a_{21}+b_{21}a_{22})b_{22}^{-1}YWX\\
&=-a_{22}XWX+(a_{11}b_{21}+a_{21}b_{22})b_{22}^{-1}YWX
\end{align*}
by (\ref{eq.comm}), and
\begin{align*}
&\frac{a_{22}b_{22}X-(a_{11}b_{21}+a_{21}b_{22})Y}{b_{11}b_{22}}(b_{22}ZY+a_{21}YW+a_{11}YZ)\\
&=-(a_{22}b_{11}^{-1}X-(a_{11}b_{21}+a_{21}b_{22})b_{11}^{-1}b_{22}^{-1}Y)(a_{22}XW+b_{11}WX+b_{21}WY)\\
&=-a_{22}^2b_{11}^{-1}XXW+(a_{11}b_{21}+a_{21}b_{22})a_{22}b_{11}^{-1}b_{22}^{-1}YXW\\
&\qquad\qquad-a_{22}XWX+(a_{11}b_{21}+a_{21}b_{22})b_{22}^{-1}YWX\\
&\qquad\qquad\qquad\qquad-a_{22}b_{21}b_{11}^{-1}XWY+(a_{11}b_{21}+a_{21}b_{22})b_{21}b_{11}^{-1}b_{22}^{-1}YWY\\
&=a_{22}^2a_{21}a_{11}^{-1}b_{11}^{-1}XYW-(a_{11}b_{21}+a_{21}b_{22})a_{22}^{2}a_{11}^{-1}b_{11}^{-1}b_{22}^{-1}XYW\\
&\qquad\qquad-a_{22}XWX+(a_{11}b_{21}+a_{21}b_{22})b_{22}^{-1}YWX\\
&\qquad\qquad\qquad\qquad+a_{22}^2b_{21}b_{11}^{-1}b_{22}^{-1}XYW+(a_{11}b_{21}+a_{21}b_{22})a_{22}b_{21}b_{11}^{-1}b_{22}^{-2}Y^2W\\
&=-a_{22}XWX+(a_{11}b_{21}+a_{21}b_{22})b_{22}^{-1}YWX,
\end{align*}
so
\[(b_{22}ZY+a_{21}YW+a_{11}YZ)X = \frac{a_{22}b_{22}X-(a_{11}b_{21}+a_{21}b_{22})Y}{b_{11}b_{22}}(b_{22}ZY+a_{21}YW+a_{11}YZ).\]
In addition, the other equations hold in the same way.
By direct calculations, we have
\[k\<X,Y,Z,W\>/(g_1,\dots, g_9 ,w) \cong S^!.\]
Moreover, by use of the Hilbert series, it follows that $w$ is regular.
\end{proof}

Note that it follows from the proof of Lemma \ref{lem.nord} that the normalizing automorphism $\nu:=\nu_w$ of the normal element  $w \in k\<X,Y,Z,W\>/(g_1, \dots, g_9)=(A\circ_{\psi} B)^!$ is given by
\begin{align*}
&\nu(X)= \frac{a_{11}b_{11}X+(a_{11}b_{21}+a_{21}b_{22})Y}{a_{11}a_{22}}\\
&\nu(Y)= \frac{b_{22}}{a_{11}}Y\\
&\nu(Z)= \frac{a_{11}b_{11}Z+(a_{11}b_{21}+a_{21}b_{22})W}{b_{11}b_{22}}\\
&\nu(W)= \frac{a_{22}}{b_{11}}W.
\end{align*}

We then calculate the algebra $C(A\circ_{\psi} B)$.
Recall that $C(A\circ_{\psi} B) = (A\circ_{\psi} B)^![w^{-1}]_0=\{aw^{-i}\mid a\in (A\circ_{\psi} B)^!_{2i}, i\in \NN\}$ is a finite dimensional algebra with the multiplication $(aw^{-i})(bw^{-j})=a\nu^i(b)w^{-i-j}$. Moreover, we have an equivalence of triangulated categories $\uCM^{\ZZ}(A\circ_{\psi} B) \overset{\cong}{\longrightarrow} \Db(\mod C(A\circ_{\psi} B))$ by Proposition \ref{prop.C(A)}.

\begin{thm}\label{thm-c}
Let $A=k[u,v], B=k[x,y]$, and let $\psi$ be as in the beginning of this subsection.
Let $w$ be the normal element of $(A\circ_{\psi} B)^!$ as in Lemma \ref{lem.nord}.
Then $C(A\circ_{\psi} B)$ is isomorphic to $M_2(k) \times M_2(k)$.
\end{thm}

\begin{proof}
The map $\rho : C(A\circ_{\psi} B) \to M_2(k) \times M_2(k)$ defined by
\begin{align*}
&1 &&\mapsto &&\left(\begin{pmatrix} 1&0 \\ 0&1 \end{pmatrix}, \begin{pmatrix} 1&0 \\ 0&1 \end{pmatrix}\right),\\
&t_1=a_{22}YWw^{-1} &&\mapsto &&\left(\begin{pmatrix} 0&1 \\ 0&0 \end{pmatrix}, \begin{pmatrix} 0&0 \\ 0&0 \end{pmatrix}\right),\\
&t_2=a_{11}YXw^{-1} &&\mapsto &&\left(\begin{pmatrix} 0&0 \\ 0&0 \end{pmatrix}, \begin{pmatrix} 0&1 \\ 0&0 \end{pmatrix}\right),\\
&t_3=Y(a_{11}Z+a_{21}W)w^{-1} &&\mapsto &&\left(\begin{pmatrix} 1&0 \\ 0&0 \end{pmatrix}, \begin{pmatrix} 1&0 \\ 0&0 \end{pmatrix}\right),\\
&t_4=W(b_{11}X+b_{21}Y)w^{-1}  &&\mapsto &&\left(\begin{pmatrix} -1&0 \\ 0&0 \end{pmatrix}, \begin{pmatrix} 0&0 \\ 0&-1 \end{pmatrix}\right),\\
&t_5=b_{11}WZw^{-1}  &&\mapsto &&\left(\begin{pmatrix} 0&0 \\ 0&0 \end{pmatrix}, \begin{pmatrix} 0&0 \\ -1&0 \end{pmatrix}\right),\\
&t_6=X(a_{11}Z+a_{21}W)w^{-1}  &&\mapsto &&\left(\begin{pmatrix} 0&0 \\ -1&0 \end{pmatrix}, \begin{pmatrix} 0&0 \\ 0&0 \end{pmatrix}\right),\\
&t_7=a_{11}^2a_{22}b_{22}^{-1}YXWZw^{-2}  &&\mapsto &&\left(\begin{pmatrix} 0&0 \\ 0&0 \end{pmatrix}, \begin{pmatrix} -1&0 \\ 0&0 \end{pmatrix}\right)
\end{align*}
is an algebra homomorphism because the multiplication table of $\{ 1,t_2, t_3,t_4,t_5,t_6,t_7\}$ given below coincides with the multiplication table of the set of the corresponding elements.
\begin{center}
\begin{tabular}{ c|c c c c c c c c}
   & $1$ & $t_1$ & $t_2$ & $t_3$ & $t_4$ & $t_5$ & $t_6$ & $t_7$\\
\hline
$1$ &$1$ & $t_1$ & $t_2$ & $t_3$ & $t_4$ & $t_5$ & $t_6$ & $t_7$\\
%\hline
$t_1$ & $t_1$ & $0$ & $0$ & $0$ & $0$ & $0$ & $-t_3-t_7$ & $0$ \\
%\hline
$t_2$ & $t_2$ & $0$ & $0$ & $0$ & $-t_2$ & $t_7$ & $0$ & $0$ \\
%\hline
$t_3$ & $t_3$ & $t_1$ & $t_2$ & $t_3$ & $-t_3-t_7$ & $0$ & $0$ & $t_7$ \\
%\hline
$t_4$ & $t_4$ & $-t_1$ & $0$ & $-t_3-t_7$ & $-t_4$ & $-t_5$ & $0$ &$0$ \\
%\hline
$t_5$ & $t_5$ & $0$ & $t_3+t_4+t_7$ & $t_5$ & $0$ & $0$ & $0$ &$-t_5$ \\
%\hline
$t_6$  & $t_6$ &$-1-t_4-t_7$ & $0$ & $t_6$ & $-t_6$ & $0$ & $0$ & $0$ \\
%\hline
$t_7$ & $t_7$ & $0$ & $-t_2$ & $t_7$ & $0$ & $0$ & $0$ &$-t_7$ \\
%\hline
\end{tabular}
\end{center}
For example,
\begin{align*}
t_1^2&= a_{22}^2YWw^{-1}YWw^{-1}
=-a_{22}b_{22}WYw^{-1}YWw^{-1}\\
&=-a_{22}b_{22}WY\nu(Y)\nu(W)w^{-2}
=-\frac{a_{22}b_{22}^2}{a_{11}}WY^2\nu(W)w^{-2}=0,
\end{align*}
\begin{align*}
t_1t_6&=a_{22}YWw^{-1}X(a_{11}Z+a_{21}W)w^{-1}
=-b_{22}WY\nu(X)\nu(a_{11}Z+a_{21}W)w^{-2}\\
&=-b_{22}WY(\frac{a_{11}b_{11}X+(a_{11}b_{21}+a_{21}b_{22})Y}{a_{11}a_{22}})\nu(a_{11}Z+a_{21}W)w^{-2}\\
&=-\frac{b_{11}b_{22}}{a_{22}}WYX\nu(a_{11}Z+a_{21}W)w^{-2}
=b_{11}YWX\nu(a_{11}Z+a_{21}W)w^{-2}\\
&=Y(-w-a_{22}XW-b_{21}WY)\nu(a_{11}Z+a_{21}W)w^{-2}\\
&=-Yw\nu(a_{11}Z+a_{21}W)w^{-2}-a_{22}YXW\nu(a_{11}Z+a_{21}W)w^{-2}\\
&=-Y(a_{11}Z+a_{21}W)w^{-1}-a_{22}YXW(\frac{a_{11}^2}{b_{22}}Z)w^{-2}=-t_3-t_7,
\end{align*}
\begin{align*}
t_2t_4&=a_{11}YXw^{-1}W(b_{11}X+b_{21}Y)w^{-1}=
a_{11}YXw^{-1}(-w- a_{22}XW)w^{-1}\\
&=a_{11}YXw^{-1}(-1- a_{22}XWw^{-1})
=-a_{11}YXw^{-1}-a_{11}a_{22}YXw^{-1}XWw^{-1}\\
&=-a_{11}YXw^{-1}-a_{11}a_{22}YX\nu(X)\nu(W)w^{-2}\\
&=-a_{11}YXw^{-1}-a_{11}a_{22}YX(\frac{a_{11}b_{11}X+(a_{11}b_{21}+a_{21}b_{22})Y}{a_{11}a_{22}})\nu(W)w^{-2}\\
&=-a_{11}YXw^{-1}-a_{11}Y(\frac{a_{11}b_{11}a_{21}YX-(a_{11}b_{21}+a_{21}b_{22})a_{11}YX}{a_{11}a_{22}})\nu(W)w^{-2}\\
&=-a_{11}YXw^{-1}=-t_2.
\end{align*}
The rest of the cases can be calculated in the similar way.
Clearly $\rho$ is a bijection.
\end{proof}

\begin{rem}
By Theorems \ref{thm-d4} and \ref{thm.rel}, we see that the twisted Segre product $A\circ_{\psi} B$ of $A=k[u,v]$ and $B=k[x,y]$ with respect to a diagonal twisting map $\psi$ is a noncommutative graded isolated singularity.
Combining of Theorem \ref{thm-c} and Proposition \ref{prop.isol} yields another proof of this fact.
\end{rem}

By Theorem \ref{thm-c} and Proposition \ref{prop.C(A)}, we obtain the following conclusion.

\begin{cor}\label{cor-sing}
Let $A=k[u,v], B=k[x,y]$, and let $\psi$ be as in the beginning of this subsection.
Then there exists an equivalence of triangulated categories
\[\uCM^{\ZZ}(A\circ_{\psi} B) \overset{\cong}{\longrightarrow} \Db(\mod k \times k).\]
\end{cor}

\section*{Acknowledgments}
The authors thank the referees for their valuable suggestions and comments.
J.-W. He was supported by NSFC No.~11971141.
K. Ueyama was supported by JSPS KAKENHI No.~JP18K13381 and JP22K03222.

% % % % % % % % % % % % % % % % % % % % % % % % % % % %
% % % % % % % % % % % % % % % % % % % % % % % % % % % %

\end{document}